\documentclass[dvipdfmx,English,11pt]{amsart}

\usepackage{tikz}
\usepackage[dvipdfmx]{hyperref}
\usepackage{amsmath,amsthm,amssymb,amsfonts}
\usepackage{bm}
\usepackage{latexsym}
\usepackage[all]{xy}
\usepackage{enumerate}
\usepackage{tcolorbox}
\usepackage{xcolor}
\usepackage{mathrsfs}
\usepackage{docmute}
\usepackage{comment}
\usepackage{nameref}
\usepackage{zref-xr}
\usepackage{graphicx}
\usepackage{here}
\usepackage{adjustbox}
\usepackage{dynkin-diagrams}
\usepackage{mathdots}
\hypersetup{
   colorlinks=true,
   citecolor=blue,
   linkcolor=blue,
   urlcolor=blue,
}
\tcbuselibrary{breakable, skins, theorems}
\usetikzlibrary{cd}

\theoremstyle{plain}
\newtheorem{thm}{Theorem}[section]
\newtheorem{prop}[thm]{Proposition}
\newtheorem{cor}[thm]{Corollary}
\newtheorem{lem}[thm]{Lemma}

\theoremstyle{definition}
\newtheorem{dfn}[thm]{Definition}

\theoremstyle{remark}
\newtheorem{rem}[thm]{Remark}
\newtheorem{ex}[thm]{Example}

\DeclareMathOperator{\Hom}{Hom}
\DeclareMathOperator{\End}{End}

\DeclareMathOperator{\gll}{\mathfrak{gl}}

\DeclareMathOperator{\qq}{\mathfrak{q}}
\DeclareMathOperator{\sqq}{\mathfrak{sq}}

\DeclareMathOperator{\diag}{diag}

\let\Im\relax
\DeclareMathOperator{\Im}{Im}

\DeclareMathOperator{\Ind}{Ind}
\DeclareMathOperator{\Res}{Res}
\DeclareMathOperator{\Coind}{Coind}
\DeclareMathOperator{\ch}{ch}

\DeclareMathOperator{\id}{id}
\DeclareMathOperator{\Ob}{Ob}

\DeclareMathOperator{\Par}{par}
\DeclareMathOperator{\HOM}{HOM}
\DeclareMathOperator{\END}{END}

\DeclareMathOperator{\hd}{hd}
\DeclareMathOperator{\soc}{soc}
\DeclareMathOperator{\height}{ht}

\newcommand{\gMod}{\text{-}\mathsf{gMod}}
\newcommand{\gmod}{\text{-}\mathsf{gmod}}

\newcommand{\sMod}{\text{-}\mathsf{sMod}}
\newcommand{\smod}{\text{-}\mathsf{smod}}
\newcommand{\Mods}{\mathsf{sMod}\text{-}}
\newcommand{\mods}{\mathsf{smod}\text{-}}
\newcommand{\sgMod}{\text{-}\mathsf{sgMod}}
\newcommand{\sgmod}{\text{-}\mathsf{sgmod}}
\newcommand{\sVect}{\text{-}\mathsf{sVect}}

\newcommand{\sAlg}{\text{-}\mathsf{sAlg}}

\renewcommand{\c}{\mathsf{c}}

\newcommand{\Triv}{\mathsf{triv}}

\begin{document}

\title{Quiver Hecke--Clifford superalgebras and R-matrices}
\author{Koreto Endo}
\address{Department of Mathematics, Institute of Science Tokyo, 2-12-1 Ookayama, Meguro-ku, Tokyo, 152-8551, Japan}
\email{endo.k.d6a6@m.isct.ac.jp}
\date{\today}
\maketitle

\begin{abstract}
   In this paper, we develop the foundations of the representation theory of quiver Hecke--Clifford superalgebras.
   We further construct a Schur--Weyl duality between quantum affine analogues of the queer Lie superalgebra and the quiver Hecke--Clifford superalgebra of type $A_{\infty}$, based on the construction due to Kwon--Lee.
\end{abstract}

\section{Introduction}

\subsection{Lie superalgebras and quantum groups of type \(Q\)}

The queer Lie superalgebra \(\qq_n\) is a distinctive object in the theory of Lie superalgebras over \(\mathbb{C}\).
According to Kac's classification \cite{K2} of simple Lie superalgebras, the subquotients of \(\qq_n\) yield the family of simple Lie superalgebras of type \(Q\).
A special feature of the representation theory of \(\qq_n\) arises from the fact that its Cartan subalgebra \(\mathfrak{h} \subset \qq_n\) contains odd elements.
Despite not belonging to the family of Kac--Moody Lie superalgebras, simple Lie superalgebras of type \(Q\) have a representation theory that bears a strong resemblance to that of the general linear Lie superalgebra \(\gll_{m|n}\).
The finite-dimensional irreducible representations of \(\qq_n\) are weight modules with respect to the even part \(\mathfrak{h}_{\overline{0}}\) of the Cartan subalgebra, and admit a classification via highest weight theory (cf. \cite{CW}).

Our focus here is on the Schur--Weyl duality for \(\qq_n\) discovered by Sergeev \cite{S}.
It is obtained from the Schur--Weyl duality for \(\gll_{n|n}\) by adjoining an odd involution on the vector representation of \(\qq_n\).
More precisely, \(\qq_n\) admits a Schur--Weyl duality with the \textit{Sergeev superalgebra}.
This superalgebra is obtained from the group algebra \(\mathbb{C}[\mathfrak{S}_d]\) of the symmetric group by adjoining Clifford generators, and it is always semisimple.
The corresponding Schur--Weyl duality fits naturally into the framework of the double supercentralizer theorem.
This theory was subsequently quantized by Olshanski \cite{O}.
He introduced the quantum group \(U_q(\qq_n)\) of \(\qq_n\) and the Hecke--Clifford superalgebra as a quantization of the Sergeev superalgebra, and established a Schur--Weyl duality between them.
The fact that the classical limit of \(U_{q}(\qq_n)\) coincides with \(U(\qq_n)\) was later proved in \cite{GJKK}.

According to \cite{GS}, there are two affine analogues of the Lie superalgebra of type \(Q\).
The first is the \textit{untwisted affine Lie superalgebra} \(\qq_n^{(1)}\), constructed from the loop Lie superalgebra by
\[\qq_n^{(1)} = \mathfrak{sq}_n \otimes \mathbb{C}[t, t^{-1}] \oplus \mathbb{C}D,\]
where \(\mathfrak{sq}_n = [\qq_n,\qq_n]\).
Since its Cartan subalgebra still contains odd elements, \(\qq_n^{(1)}\) is not a Kac--Moody Lie superalgebra.
The \textit{twisted affine Lie superalgebra} \(\qq_n^{(2)}\) is obtained as a subquotient of \(\qq_n^{(1)}\) via an automorphism of \(\mathfrak{sq}_n\).
Remarkably, \(\qq_n^{(2)}\) is a (non-symmetrizable) Kac--Moody Lie superalgebra (cf. \cite{HS}).
Its construction was motivated by the classification of finite-growth Lie superalgebras, and \(\qq_n^{(2)}\) provides one of the rare examples of a non-symmetrizable Kac--Moody Lie superalgebra of finite growth.
By Mathieu \cite{M1}, the only simple \(\mathbb{Z}\)-graded Lie algebras of finite growth that are related to Kac--Moody Lie algebras are simple Lie algebras and loop Lie algebras.
The existence of \(\qq_n^{(2)}\) is therefore of considerable interest.
A ``quantum affine superalgebra'' \(U_{q}(\widehat{\qq}_n^{\mathsf{tw}})\) associated with the Lie superalgebra of type \(Q\) was further introduced by Chen--Guay \cite{CG}.
The classical limit of \(U_{q}(\widehat{\qq}_n^{\mathsf{tw}})\) does not coincide with \(U(\qq_n^{(2)})\).
Nonetheless, an evaluation homomorphism to \(U_{q}(\qq_n)\) exists \cite{LMZ}.
Moreover, a Schur--Weyl duality with the affine Hecke--Clifford superalgebra has been established, providing an affine analogue of Olshanski's result.

\subsection{Quiver Hecke(--Clifford) superalgebras}

The quiver Hecke algebra (or Khovanov--Lauda--Rouquier algebra) was introduced independently by Khovanov--Lauda \cite{KL1,KL2} and Rouquier \cite{R}.
It categorifies the negative half of the quantum group of a symmetrizable Kac--Moody Lie algebra.
The Schur--Weyl duality between an arbitrary quantum affine algebra and a symmetric quiver Hecke algerbas was established by Kang--Kashiwara--Kim \cite{KKK}.

Kang--Kashiwara--Tsuchioka \cite{KKT} introduced two super analogues of the quiver Hecke algebra.
The more direct generalization is the quiver Hecke superalgebra, obtained by equipping generators of the same type as those of the quiver Hecke algebra with a non-trivial \(\mathbb{Z}/2\mathbb{Z}\)-grading and modifying the defining relations accordingly.
This was subsequently employed in \cite{KKO1,KKO2} for the categorification of quantum affine (super)algebras.
The other is the quiver Hecke--Clifford superalgebra, obtained from the quiver Hecke algebra by adjoining Clifford generators in the spirit of the construction of Sergeev algebra.
In analogy with Brundan--Kleshchev's isomorphism \cite{BK} between the cyclotomic quiver Hecke algebra and the cyclotomic affine Hecke algebra, Kang--Kashiwara--Tsuchioka further established an isomorphism between the cyclotomic quiver Hecke--Clifford superalgebra and the cyclotomic affine Hecke--Clifford superalgebra.
By analogy with the result of Chen--Guay, this suggests the existence of a Schur--Weyl duality between a ``quantum affine superalgebra of type \(Q\)'' and the quiver Hecke--Clifford superalgebra.

\subsection{Main results}

The aim of this paper is to merge the constructions of Kang--Kashiwara--Kim and Sergeev so as to build a Schur--Weyl duality between ``quantum affine analogues of \(\qq_n\)'' and a symmetric quiver Hecke--Clifford superalgebra.
To this end, we first develop the foundations of the representation theory of symmetric quiver Hecke--Clifford superalgebras \(RC_{\beta}\).

For \(M_i \in RC_{\beta_i} \sgMod\), we define the \textit{convolution product} \(M_1 \circ M_2 \in RC_{\beta_1 + \beta_2} \sgMod\) by
   \[M_1 \circ M_2 = RC_{\beta_1 + \beta_2}e(\beta_1,\beta_2) \otimes_{(RC_{\beta_1} \otimes RC_{\beta_2})}(M_1 \otimes M_2).\]
As in the case of the quiver Hecke algebra, we introduce R-matrices between convolution products.
Specifically, for \(1 \leq a < |\beta|\), we define the \textit{intertwiner} \(\varphi_a \in RC_{\beta}\) and use it to construct the \textit{R-matrix} \(R_{M,N} \colon M \circ N \to N \circ M\) for \(M \in RC_{\beta_1} \sgMod\) and \(N \in RC_{\beta_2} \sgMod\) (see Section \ref{subsec:intertwiners}).
In the symmetric case, one can further define the \textit{affinization} and the \textit{renormalized R-matrix} of a module.
The following proposition describes the direct relationship between the quiver Hecke algebra and the quiver Hecke--Clifford superalgebra arising from the same symmetric Cartan superdatum.

\begin{prop}[Proposition \ref{prop:morita equivalence for quiver Hecke--Clifford superalgebra}]
   Let \(R_{\beta}\) and \(RC_{\beta}\) be the symmetric quiver Hecke algebra and the symmetric quiver Hecke--Clifford superalgebra constructed from the same Cartan superdatum.

   Let \(F_{\beta} = e^{\dagger}(\beta).- \colon RC_{\beta} \sgMod \to R_{\beta} \sgMod\) denote the equivalence of categories from \cite{KKT}.
   For any \(M_i \in RC_{\beta_i} \sgMod\) (\(i = 1,2\)), there is an isomorphism
   \[F_{\beta_1 + \beta_2}(M_1 \circ M_2) \simeq F_{\beta_1}(M_1) \circ F_{\beta_2}(M_2).\]
   Equivalently, \(F \coloneqq \bigoplus_{\beta \in \mathsf{Q}^+} F_{\beta}\) is a monoidal functor.
\end{prop}

By means of this equivalence, much of the the results in representation theory of symmetric quiver Hecke--Clifford superalgebras can be deduced from that of the corresponding quiver Hecke algebras.
For example, the following holds.

\begin{cor}[{Corollary \ref{cor:real simple module of quiver Hecke--Clifford superalgebra}}]
   Let \(RC_{\beta_i}\) (\(i= 1,2\)) be symmetric quiver Hecke--Clifford superalgebras, and let \(M_i \in RC_{\beta_i} \sgmod\) be simple modules, at least one of which is real simple.
   \begin{enumerate}[(i)]
      \item If \(\hd M \circ N \simeq \soc N \circ M\), then \(M \circ N\) is simple and \(M \circ N \simeq N \circ M\).
      \item \(M \circ N \simeq N \circ M\) if and only if \(M \circ N\) is simple.
   \end{enumerate}
\end{cor}

As an application, we construct a Schur--Weyl duality via quiver Hecke--Clifford superalgebras.
As an analogue of Sergeev's method, our aim is to produce such a duality by adding an odd involution to an existing Schur--Weyl duality for a ``quantum affine superalgebra of type \(A\)'' realized through a symmetric quiver Hecke algebra.
For this purpose, we draw on the Schur--Weyl duality due to Kwon--Lee \cite{KL4} between the generalized quantum group \(\mathcal{U}(\epsilon)\) of type \(A\) and the quiver Hecke algebra of type \(A_{\infty}\).
\begin{thm}[Theorem \ref{thm:construction of bimodule of untwisted version}, Theorem \ref{thm:construction of bimodule of twisted version}]
   Fix an arbitrary odd involution \(P\) on the affinization \(V_z\) of the vector representation of \(\mathcal{U}(\epsilon)\).
   By a suitable choice of subsuperalgebras \(\mathcal{U}(\epsilon)^{P}\) and \(\mathcal{U}^{\mathsf{tw}}(\epsilon)^{P}\) of \(\mathcal{U}(\epsilon)\), one obtains a \((\mathcal{U}(\epsilon)^{P},RC_{\beta})\)-bisupermodule \(V'^{\otimes \beta}_{\mathbb{O}}\) and a \((\mathcal{U}^{\mathsf{tw}}(\epsilon)^{P},RC_{\beta})\)-bisupermodule \(V^{\mathsf{tw} \otimes \beta}_{\mathbb{O}}\).
\end{thm}
We consider two natural choices of odd involution.
The first is an odd morphism on the affinization of the vector representation of \(\qq_n^{(2)}\), and the second is the analogue of the odd morphism on the affinization of the vector representation of \(U_{q}(\widehat{\qq}_n^{\mathsf{tw}})\).

We note that the subalgebra of \(\mathcal{U}(\epsilon)\) with which the Schur--Weyl duality is established depends on the choice of odd involution used in the construction.

In Section \ref{subsec:Properties of the Schur--Weyl functors}, we investigate the properties of the Schur--Weyl functors arising from these bimodules.
In analogy with \cite{KKK,KL4}, exactness holds for these functors.
\begin{thm}[Theorem \ref{thm:exactness of SW functor for queer}]
   The functors \(V'^{\otimes \beta}_{\mathbb{O}} \otimes_{RC_{\beta}} -\) and \(V^{\mathsf{tw} \otimes \beta}_{\mathbb{O}} \otimes_{RC_{\beta}} -\) are exact.
\end{thm}
We also verify that these functors do not vanish by computing their values on segment modules.
\begin{prop}[Example \ref{ex:segment module}]
   For the length-\(1\) segment module \(L(a) \in RC_{\alpha_a}\) with \(a \in \mathbb{Z}\), the following isomorphisms hold:
   \begin{align*}
      V'^{\otimes \beta}_{\mathbb{O}} \otimes_{RC_{\beta}} L(a) &\simeq \mathcal{W}_{1,\epsilon}, \\
      V^{\mathsf{tw} \otimes \beta}_{\mathbb{O}} \otimes_{RC_{\beta}} L(a) &\simeq \mathcal{W}_{1,\epsilon}.
   \end{align*}
\end{prop}
Note that we can not discuss the monoidality of these functors naively.
Indeed, \(\mathcal{U}(\epsilon)^{P}\) and \(\mathcal{U}^{\mathsf{tw}}(\epsilon)^{ P}\) may not be Hopf subsuperalgebras of \(\mathcal{U}(\epsilon)\), so the tensor product of modules over these algebras may not be defined.

\subsection{Organization}

The paper is organized as follows.
In Section \ref{sec:preliminaries}, we reviews the basic properties of superalgebras and their representation theory, as well as the formalism of bisupermodules and Morita superequivalence.
In Section \ref{sec:quiver hecke--clifford superalgebras}, we develops the foundational theory of (symmetric) quiver Hecke--Clifford superalgebras.
Specifically, we define and study the convolution product, dual modules, affinization, and (renormalized) R-matrices for representations of quiver Hecke--Clifford superalgebras.
In Section \ref{sec:schur--weyl duality via quiver hecke--clifford superalgebra}, we carries out the construction of a Schur--Weyl duality via quiver Hecke--Clifford superalgebras and discusses properties of the resulting Schur--Weyl functors in analogy with \cite{KKK}.

\section{Preliminaries}\label{sec:preliminaries}

In this section, we review the representation theory of superalgebras, following \cite{BE,KKT,K2, KL3}.

Throughout, let \(\Bbbk\) be a field with \(\ch \Bbbk \neq 2\) and \(\sqrt{-1} \in \Bbbk\).
A \textit{super vector space} over \(\Bbbk\) is a \(\mathbb{Z}/2\mathbb{Z}\)-graded \(\Bbbk\)-vector space
\[V = V_{\overline{0}} \oplus V_{\overline{1}}.\]
The \textit{parity} of \(v \in V_{\overline{i}}\) is defined by \(\Par(v) \coloneqq \overline{i}\).
An element \(v \in V\) is said to be \textit{even} (resp.\ \textit{odd}) if \(\Par(v) = \overline{0}\) (resp.\ \(\Par(v) = \overline{1}\)).

Let \(\Bbbk \sVect\) denote the category of super vector spaces over \(\Bbbk\) together with \(\Bbbk\)-linear maps.
For \(V, W \in \Bbbk \sVect\), the space \(\Hom(V,W)\) carries a natural super vector space structure
\begin{align*}
\Hom(V,W) &= \Hom(V,W)_{\overline{0}} \oplus \Hom(V,W)_{\overline{1}}, \\
\Hom(V,W)_{\overline{0}} &= \{f \in \Hom(V,W) \mid f(V_{\overline{i}}) \subset W_{\overline{i}}~(i = 0,1)\}, \\
\Hom(V,W)_{\overline{1}} &= \{f \in \Hom(V,W) \mid f(V_{\overline{i}}) \subset W_{\overline{i+1}}~(i = 0,1)\}.
\end{align*}

Equipped with the usual tensor product, the category \(\Bbbk \sVect\) is symmetric monoidal, with braiding given by the Koszul sign rule
\[T_{V,W} \colon V \otimes W \to W \otimes V; \quad v \otimes w \mapsto (-1)^{\Par(v)\Par(w)}w \otimes v.\]
All subsequent constructions are understood relative to this braiding.

\subsection{Superalgebras}

A super vector space \(A\) is called a \textit{superalgebra} if \(A\) is an algebra satisfying
\[A_{\overline{i}}A_{\overline{j}} \subset A_{{\overline{i+j}}}.\]

The \textit{underlying algebra} \(|A|\) of a superalgebra \(A\), obtained by forgetting the \(\mathbb{Z}/2\mathbb{Z}\)-grading, is an ordinary algebra.
Conversely, any \(\Bbbk\)-algebra \(A\) gives rise to a superalgebra \(A^{\Triv} = A_{\overline{0}} \oplus A_{\overline{1}}\) by setting \(A_{\overline{0}} = A\) and \(A_{\overline{1}} = 0\).

\begin{ex}
The \textit{Clifford superalgebra} is the (simple) superalgebra \(\mathcal{C}_{n}\) over \(\Bbbk\) defined by
\[\mathcal{C}_n = \langle \c_1, \ldots, \c_n \rangle / \langle \c_i \c_j + \c_j \c_i = 2\delta_{ij}~(1 \leq i,j \leq n) \rangle,\]
\[\Par(\c_i) = \overline{1}~(1 \leq i \leq n).\]
\end{ex}

Let \(A, B\) be superalgebras.
A linear map \(f \colon A \to B\) is called a \textit{morphism of superalgebras} if \(f \colon |A| \to |B|\) is an even algebra homomorphism.
The category of superalgebras over \(\Bbbk\) and their morphisms is denoted by \(\Bbbk \sAlg\).

Let \(A \in \Bbbk \sAlg\).
A super vector space \(M \in \Bbbk \sVect\) is called a \textit{left super \(A\)-module} if \(M\) is a left \(|A|\)-module satisfying \(A_{\overline{i}}.M_{\overline{j}} \subset M_{\overline{i+j}}\).
Similarly, \(M\) is called a \textit{right super \(A\)-module} if \(M\) is a right \(|A|\)-module satisfying \(M_{\overline{j}}.A_{\overline{i}} \subset M_{\overline{i+j}}\).
The left super \(A\)-module structures on \(M\) are in one-to-one correspondence with superalgebra morphisms \(\rho \colon A \to \End_{\Bbbk}(M)\).
Indeed, if \(M\) is a left super \(A\)-module, then the map \(\rho \colon A \to \End_{\Bbbk}(M)\) given by
\[\rho(a) = a.-\]
is a superalgebra morphism.
Conversely, given such a morphism \(\rho\), the rule
\[a.m = \rho(a)(m)\]
endows \(M\) with a super \(A\)-module structure.
These two constructions are mutually inverse.

Now let \(M, N\) be left super \(A\)-modules.
A \(\Bbbk\)-linear map \(f \colon M \longrightarrow N\) is called a \textit{morphism of left super \(A\)-modules} if it supercommutes with the action of \(A\).
More precisely,
\[f(ax) = (-1)^{\Par(f)\Par(a)}af(x)\]
holds for all \(a \in A\) and \(x \in M\).
Likewise, when \(M, N\) are right super \(A\)-modules, a linear map \(f \colon M \to N\) is called a \textit{morphism of right super \(A\)-modules} if \(f(xa) = f(x)a\) for all \(a \in A\) and \(x \in M\).

The category of left super \(A\)-modules and their morphisms is denoted by \(A \sMod\), and its full subcategory consisting of finite-dimensional super \(A\)-modules by \(A \smod\).
For \(M, N \in A \sMod\), let \(\Hom_A(M, N)\) denote the space of all module morphisms from \(M\) to \(N\).
Setting
\[\Hom_{A}(M,N)_{\overline{i}} = \Hom_{A}(M,N) \cap \Hom_{\Bbbk}(M,N)_{\overline{i}},\]
the space
\[\Hom_{A}(M,N) = \Hom_{A}(M,N)_{\overline{0}} \oplus \Hom_{A}(M,N)_{\overline{1}}\]
carries a natural super vector space structure.
The category \(\Mods A\) of right super \(A\)-modules and its full subcategory \(\mods A\) of finite-dimensional modules are defined analogously.

\begin{rem}
   Let \(M, N \in A \sMod\) and \(f \in \Hom_{A}(M,N)\).
   The linear map
   \[f^{\dagger} \colon M \to N; \quad x \mapsto (-1)^{\Par(f)\Par(x)}f(x)\]
   satisfies \(f^{\dagger}(ax) = af^{\dagger}(x)\) for all \(a \in A\).
   In other words, the involution
   \[(-)^{\dagger} \colon \Hom_{\Bbbk}(M,N) \to \Hom_{\Bbbk}(M,N)\]
   induces a bijection between \(\Hom_{A}(M,N)\) and \(\Hom_{|A|}(M,N)\).
\end{rem}

For \(A \in \Bbbk \sAlg\) and \(M \in A \sMod\), the \textit{parity change module} \(\Pi(M)\) of \(M\) is defined by
\begin{align*}
&\Pi(M) = \Pi(M)_{\overline{0}} \oplus \Pi(M)_{\overline{1}},~\Pi(M)_{\overline{i}} = \{\pi(x) \mid x \in M_{{\overline{i+1}}}\}, \\
&a.\pi(x) = \pi((-1)^{\Par(a)}a.x) ~(a \in A, x \in M).
\end{align*}
The assignment \(\Pi\) defines an auto-equivalence on \(A \sMod\) and \(A \smod\), called the \textit{parity change functor}.

\begin{rem}
Define an auto-equivalence \(\Pi'\) on \(A \sMod\) and \(A \smod\) by
\begin{align*}
&\Pi'(M) = \Pi'(M)_{{\overline{0}}} \oplus \Pi'(M)_{\overline{1}},~\Pi'(M)_{\overline{i}} = \{\pi'(x) \mid x \in M_{\overline{i+1}}\}, \\
&a.\pi'(x) = \pi'(a.x) ~(a \in A,~x \in M).
\end{align*}
An isomorphism \(\Pi' \simeq \Pi\) is provided by the even natural transformation \(\phi\) given by
\[\phi_M \colon \Pi'(M) \to \Pi(M); \quad \pi'(x) \mapsto (-1)^{\Par(x)}\pi(x).\]
\end{rem}

For \(A, B \in \Bbbk \sAlg\), the product \(m_{A \otimes B}\) on the super vector space \(A \otimes B\) is defined in terms of the products \(m_A, m_B\) of \(A\) and \(B\) by
\[m_{A \otimes B} \colon A \otimes B \otimes A \otimes B \xrightarrow{A \otimes T_{B,A} \otimes B} A \otimes A \otimes B \otimes B \xrightarrow{m_{A} \otimes m_{B}} A \otimes B;\]
explicitly,
\[(a_1 \otimes b_1).(a_2 \otimes b_2) = (-1)^{\Par(b_1)\Par(a_2)}a_1 a_2 \otimes b_1 b_2.\]
Under this product, \(A \otimes B\) becomes a superalgebra over \(\Bbbk\), called the \textit{tensor product algebra} of \(A\) and \(B\).

\begin{rem}
Given \(A_1, A_2 \in \Bbbk \sAlg\) and \(M_i \in A_i \sMod\) (\(i=1,2\)), the tensor product \(M_1 \otimes M_2\) carries a natural left super \(A_1 \otimes A_2\)-module structure.
More precisely, for representations \(\rho_i \colon A_i \to \End(M_i)\) (\(i=1,2\)), the corresponding representation \(\rho \colon A_1 \otimes A_2 \to \End(M_1 \otimes M_2)\) is given by the composition
\[\rho \colon A_1 \otimes A_2 \xrightarrow{\rho_1 \otimes \rho_2} \End(M_1) \otimes \End(M_2) \xrightarrow{\iota} \End(M_1 \otimes M_2),\]
where \(\iota\) is defined by
\[\iota(f \otimes g)(m_1 \otimes m_2) = (-1)^{\Par(g)\Par(m_1)}f(m_1) \otimes g(m_2).\]
The resulting action of \(A_1 \otimes A_2\) reads
\[(a_1 \otimes a_2).(m_1 \otimes m_2) \coloneqq (-1)^{\Par(a_2)\Par(m_1)}(a_1.m_1) \otimes (a_2.m_2).\]
\end{rem}

\subsection{Bisupermodules}

Let \(A \in \Bbbk \sAlg\) with product denoted by \(\cdot\).
A new superalgebra structure on \(A\) arises from the product
\[a \cdot_{\mathsf{op}} b \coloneqq (-1)^{\Par(a)\Par(b)}b \cdot a,\]
that is, \(\cdot_{\mathsf{op}} = \cdot \circ T_{A,A} \colon A \otimes A \to A\).
The superalgebra \((A, \cdot_{\mathsf{op}})\), denoted by \(A^{\mathsf{op}}\), is called the \textit{opposite superalgebra} of \(A\).

\begin{rem}
With this definition, the right super \(A\)-module structures and the left super \(A^{\mathsf{op}}\)-module structures on a super vector space \(M\) are in one-to-one correspondence.
Concretely, if \(M\) is a right super \(A\)-module with action denoted by \(.\), then the rule
\[a \star m \coloneqq (-1)^{\Par(a)\Par(m)}m.a\]
endows \(M\) with a left super \(A^{\mathsf{op}}\)-module structure \((M, \star)\).
Conversely, if \((M, \star)\) is a left super \(A^{\mathsf{op}}\)-module, then
\[m.a \coloneqq (-1)^{\Par(a)\Par(m)}a \star m\]
defines a right super \(A\)-module structure on \(M\).
These constructions are mutually inverse.
From now on, we identify right super \(A\)-modules with left super \(A^{\mathsf{op}}\)-modules via this correspondence.
\end{rem}

\begin{dfn}\label{dfn:antiautomorphism}
Let \(A \in \Bbbk \sAlg\).
A linear isomorphism \(\sigma \colon A \to A\) is called an \textit{antiautomorphism} if
\[\sigma(a \cdot b) = \sigma(b) \cdot_{\mathsf{op}} \sigma(a) = (-1)^{\Par(a)\Par(b)}\sigma(b) \cdot \sigma(a)\]
holds for all \(a, b \in A\).
\end{dfn}

Given an antiautomorphism \(\sigma\) of \(A\), any right super \(A\)-module \(M\) may be regarded as a left super \(A\)-module via
\[a.m = (-1)^{\Par(a)\Par(m)}m.\sigma(a)~(a \in A, m \in M).\]

\begin{dfn}
Let \(A, B \in \Bbbk \sAlg\), and let \(M\) be both a left super \(A\)-module and a right super \(B\)-module.
We say that \(M\) is an \((A, B)\)-\textit{bisupermodule} if
\[(a.m).b = a.(m.b)\]
holds for all \(a \in A\), \(b \in B\), and \(m \in M\).
\end{dfn}

\begin{rem}
Let \(A, B\) be superalgebras, and let \(M\) be an \((A, B)\)-bisupermodule.
Rewriting the right action of \(B\) as a left action of \(B^{\mathsf{op}}\), the actions of \(A\) and \(B^{\mathsf{op}}\) supercommute in \(\End(M)\).
Indeed, for all \(a \in A\), \(b \in B\), and \(m \in M\),
\begin{align*}
b \star(a.m) &= (-1)^{\Par(b)\Par(a.m)}(a.m).b \\
&= (-1)^{\Par(b)(\Par(a) + \Par(m))}a.(m.b) \\
&= (-1)^{\Par(a)\Par(b)}a.(b \star m).
\end{align*}
In particular, \(M\) carries a left super \(A \otimes B^{\mathsf{op}}\)-module structure.
Conversely, any left super \(A \otimes B^{\mathsf{op}}\)-module \(M\) is in particular a left super \(B^{\mathsf{op}}\)-module and hence a right super \(B\)-module whose right action commutes with that of \(A\).
Thus \(M\) is an \((A, B)\)-bisupermodule structure.
From now on, we identify \((A, B)\)-bisupermodules with left super \(A \otimes B^{\mathsf{op}}\)-modules via this correspondence.
\end{rem}

For a superalgebra \(B\), a left super \(B\)-module \(M\), and a right super \(B\)-module \(N\), the tensor product
\[M \otimes_B N\]
is defined as the tensor product of the underlying \(|B|\)-modules and inherits a natural super vector space structure.
Moreover, if \(A\) is a superalgebra and \(M\) is an \((A, B)\)-bisupermodule, then \(M \otimes_B N\) carries a left super \(A\)-module structure, and tensoring with \(M\) yields a functor
\[M \otimes_B - \colon B \sMod \to A \sMod.\]

A right super \(B\)-module \(L \in \Mods B\) is called a \textit{flat supermodule} if the functor
\[L \otimes_B - \colon B \sMod \to \Bbbk \sMod\]
is exact.
Equivalently, for every injective morphism \(f \colon M \to N\) of super \(B\)-modules, the induced map
\[f \otimes 1 \colon L \otimes_B M \to L \otimes_B N\]
is again injective.
This holds if and only if \(L\) is flat as a right \(|B|\)-module.

Two superalgebras \(A, B \in \Bbbk \sAlg\) are said to be \textit{Morita superequivalent} if there exists an \((A, B)\)-bisupermodule \(M\) that induces a Morita equivalence between the underlying algebras \(|A|\) and \(|B|\).

\begin{ex}[{\cite[Remark 2.10(ii)]{KKT}}]\label{ex:full even idempotent}
Let \(A \in \Bbbk \sAlg\).
An element \(e \in A\) is called a \textit{full even idempotent} if
\begin{enumerate}[(i)]
\item \(e \in A_{\overline{0}}\),
\item \(e^2 = e\),
\item \(AeA = A\).
\end{enumerate}
In this case, the following mutually quasi-inverse equivalences induce a Morita superequivalence between \(A\) and \(eAe\):
\begin{align*}
e.- &\colon A \sMod \to eAe \sMod, \\
Ae \otimes_{eAe} - &\colon eAe \sMod \to A \sMod.
\end{align*}
\end{ex}

\section{Quiver Hecke--Clifford superalgebras}\label{sec:quiver hecke--clifford superalgebras}

In this section, we develop the foundations of the representation theory of quiver Hecke--Clifford superalgebras.
Since our motivation lies in Schur--Weyl duality, the representation theory of symmetric quiver Hecke--Clifford superalgebras is of particular importance.
We therefore discuss it in detail, drawing on analogies with \cite{KKK} and related works.
The readers mainly interested in the construction of Schur--Weyl duality in Section \ref{sec:schur--weyl duality via quiver hecke--clifford superalgebra} may skip Sections \ref{subsec:graded dual}, and \ref{subsec:symmetric quiver Hecke--Clifford superalgebras}.

\subsection{Quiver Hecke(--Clifford) superalgebras}\label{subsec:quiver Hecke--Clifford superalgebras}

We recall the definitions of the quiver Hecke and quiver Hecke--Clifford superalgebras introduced in \cite{KKT}.

Let \(I\) be an index set.
An integer matrix \(A = [a_{ij}]_{i,j \in I}\) is called a \textit{generalized Cartan matrix} if:
\begin{enumerate}[(i)]
\item \(a_{ii} = 2~(i \in I)\),
   \item \(a_{ij} \leq 0~(i \neq j)\),
   \item \(a_{ij} = 0 \Leftrightarrow a_{ji} = 0\).
\end{enumerate}
The matrix \(A\) is said to be \textit{symmetrizable} if there exists a diagonal matrix \(D = \diag[d_i \in \mathbb{Z}_{> 0} \mid i \in I]\) such that \(DA\) is symmetric.

A \textit{Cartan datum} is a quintuple \((A,P,\varPi,\varPi^{\vee})\) consisting of:
\begin{enumerate}[(i)]
   \item \(A\): a symmetrizable generalized Cartan matrix,
   \item \(P\): a free abelian group (called the \textit{weight lattice}),
   \item \(\varPi = \{\alpha_i \in P \mid i \in I\}\): a linearly independent subset of \(P\) (called the \textit{set of simple roots}),
   \item \(\varPi^{\vee} = \{h_i \mid i \in I\} \subset P^{\vee} = \Hom_{\mathbb{Z}}(P,\mathbb{Z})\) (called the \textit{set of simple coroots}),
   \item \(\langle h_i,\alpha_j \rangle = a_{ij}~(i,j \in I)\), where \(\langle -,- \rangle \colon P^{\vee} \times P \to \mathbb{Z}\) is the natural pairing.
\end{enumerate}

The subgroup \(\mathsf{Q} = \bigoplus_{i \in I}\mathbb{Z}\alpha_i \subset P\) is called the \textit{root lattice}.
Set \(\mathsf{Q}^{+} = \sum_{i \in I}\mathbb{Z}_{\geq 0}\alpha_i\).
For \(\beta = \sum_{i \in I} m_i \alpha_i \in \mathsf{Q}^{+}\), its \textit{height} is defined by \(|\beta| = \sum_{i \in I} m_i\).

\begin{dfn}[{\cite[\(\S 3.6\)]{KKT}}]
   A pair consisting of a Cartan datum \((A,P,\varPi,\varPi^{\vee})\) and a decomposition \(I = I_{\mathsf{even}} \sqcup I_{\mathsf{odd}}\) is called a \textit{Cartan superdatum} if
   \[a_{ij} \in 2 \mathbb{Z} ~(i \in I_{\mathsf{odd}}).\]
\end{dfn}
   For \(i,j \in I\), let \(S_{ij} \subset \mathbb{Z}^2\) denote the set of pairs \((r,s) \in \mathbb{Z}^2\) satisfying:
   \begin{enumerate}[(i)]
      \item \(0 \leq r \leq -a_{ij},~0 \leq s \leq -a_{ji}\),
      \item \(a_{ji}r + a_{ij}s = -a_{ij}a_{ji}\),
      \item \(r \in 2\mathbb{Z} ~(i \in I_{\mathsf{odd}})\),
      \item \(s \in 2\mathbb{Z} ~(j \in I_{\mathsf{odd}})\).
   \end{enumerate}
   Let \(\{t_{i,j;s,t}\}_{i \neq j, (s,t) \in S_{ij}}\) be a family of scalars satisfying:
   \begin{enumerate}[(a)]
      \item \(t_{i,j;r,s} = t_{j,i;s,r}\),
      \item \(t_{i,j;-a_{ij},0} \in \Bbbk^{\times}\).
   \end{enumerate}
   The associated family of polynomials \(\{\mathcal{Q}_{i,j}(u,v)\}_{i,j \in I}\) is defined by
   \[\mathcal{Q}_{ij}(w,z) = \cdots\]
   where we set \(\mathcal{Q}_{ij} = 0\) for \(i = j\).
   For \(\beta \in \mathsf{Q}^{+}\) with \(|\beta| = n\), set
   \[I^{\beta} = \{\nu = (\nu_1,\ldots,\nu_n) \in I^n \mid \sum_{i=1}^{n} \alpha_{\nu_i} = \beta\}.\]
\begin{dfn}[{\cite[Definition 3.10]{KKT}}]
   Given the data \((I,A,\mathcal{Q})\) and \(\beta \in \mathsf{Q}^{+}\), the \textit{quiver Hecke superalgebra} \(R_{\beta}\) is the \(\Bbbk\)-superalgebra generated by the elements \(\{e(\nu)\}_{\nu \in J^{\beta}}\),
   \(\{x_i\}_{1 \leq i \leq n}\),
   \(\{\sigma_{k}\}_{1 \leq k \leq n-1}\) with parity
   \[\Par(e(\nu)) = 0, \Par(x_i e(\nu)) = \Par(\nu_i), \Par(\sigma_k e(\nu)) = \Par(\nu_k) \Par(\nu_{k+1}) \]
   subject to the following relations:
   \begin{enumerate}[(R1)]
      \item \(e(\nu)e(\mu) = \delta_{\nu,\mu}e(\nu), \sum_{\nu \in I^{\beta}}e(\nu) = 1\),
      \item \(x_i x_j e(\nu) = \begin{cases}
         -x_j x_i e(\nu) & \text{if $i \neq j$ and $\nu_i,\nu_j \in I_{\mathsf{odd}}$},\\
         x_j x_i e(\nu) & \text{otherwise,}
      \end{cases}\)
      \item \(\sigma_{k} x_i e(\nu) = (-1)^{\Par(\nu_i)\Par(\nu_k)\Par(\nu_{k+1})}x_i \sigma_k e(\nu)\) if \(i \neq k,k+1\),
      \item \begin{align*}
         &(\sigma_{k}x_{k+1} - (-1)^{\Par(\nu_k)\Par(\nu_{k+1})}x_{k}\sigma_{k})e(\nu)\\
         &= (x_{k+1}\sigma_{k} - (-1)^{\Par(\nu_k)\Par(\nu_{k+1})}\sigma_{k}x_{k})e(\nu) = \begin{cases}
            e(\nu) & \text{if $\nu_k = \nu_{k+1}$,} \\
            0 & \text{otherwise,}
         \end{cases}
      \end{align*}
      \item \(\sigma_{k}^2e(\nu) = Q_{\nu_k,\nu_{k+1}}(x_k,x_{k+1})e(\nu)\),
      \item \(\sigma_{k}\sigma_{l}e(\nu) = (-1)^{\Par(\nu_k)\Par(\nu_{k+1}\Par(\nu_l)\Par(\nu_{l+1}))}\sigma_{l}\sigma_{k}e(\nu)\) if \(|k-l| > 1\),
      \item \begin{align*}
         &(\sigma_k \sigma_{k+1} \sigma_k - \sigma_{k+1} \sigma_k \sigma_{k+1})e(\nu) \\
         &= \begin{cases}
            \frac{Q_{\nu_k,\nu_{k+1}}(x_{k+2,x_{k+1}}) - Q_{\nu_{k},\nu_{k+1}}(x_k,x_{k+1})}{x_{k+2} - x_{k}}e(\nu) & \text{if $\nu_k = \nu_{k+2} \in I_{\mathsf{even}}$,} \\
            (-1)^{\Par(\nu_{k+1})}(x_{k+2} - x_{k}) \frac{Q_{\nu_k,\nu_{k+1}}(x_{k+2,x_{k+1}}) - Q_{\nu_{k},\nu_{k+1}}(x_k,x_{k+1})}{x_{k+2}^2 - x_{k}^2}e(\nu) & \text{if $\nu_k = \nu_{k+2} \in I_{\mathsf{odd}}$,} \\
            0 & \text{otherwise.}
         \end{cases}
      \end{align*}
   \end{enumerate}
\end{dfn}

If \(I_{\mathsf{odd}} = \varnothing\), then all generators of \(R_{\beta}\) are even, and \(R_{\beta}\) reduces to the usual quiver Hecke algebra.

   Define a new index set \(J\) by
   \[J = I_{\mathsf{even}} \times \{0,1\} \sqcup I_{\mathsf{odd}} \times \{0\},\]
   and, for \(i,j \in J\), a family of polynomials \(\tilde{\mathcal{Q}}_{i,j}\) by
   \[\tilde{\mathcal{Q}}_{i,j}(u,v) = \mathcal{Q}_{\pi_1(i),\pi_1(j)}((-1)^{\pi_2(i)}u,(-1)^{\pi_2(j)}v) \in \Bbbk[u^{1 + \pi_2(i)},v^{1 + \pi_2(j)}].\]
   For \(\beta \in \mathsf{Q}^{+}\) with \(|\beta| = n\), set
   \[J^{\beta} = \{\nu = (\nu_1,\ldots,\nu_n) \in J^n \mid \sum_{i=1}^{n} \alpha_{\pi_1(\nu_i)} = \beta\}.\]

\begin{dfn}[{\cite[Definition 3.5]{KKT}}]
   Given the data \((J,A,\tilde{\mathcal{Q}})\) and \(\beta \in \mathsf{Q}^{+}\), the \textit{quiver Hecke--Clifford superalgebra} \(RC_{\beta}\) is the \(\Bbbk\)-superalgebra generated by the even elements \(\{e(\nu)\}_{\nu \in J^{\beta}}\),
   \(\{x_i\}_{1 \leq i \leq n}\),
   \(\{\sigma_{k}\}_{1 \leq k \leq n-1}\)
   and the odd elements \(\{\c_i\}_{1 \leq i \leq n}\), subject to the following relations:
   \begin{enumerate}[(RC1)]
      \item\(e(\mu)e(\nu) = \delta_{\mu,\nu}e(\mu),~1 = \sum_{\nu \in J^{\beta}} e(\nu),~x_{p}e(\nu) = e(\nu)x_{p},~\c_{p}e(\nu) = e(c_p(\nu))\c_p,\)
      \item \(x_{p}x_{q}e(\nu) = x_{q}x_{p}e(\nu),~\c_{p}\c_{q} + \c_{q}\c_{p} = 2\delta_{pq},\)
      \item \(\c_{p}x_{q} = (-1)^{\delta_{pq}}x_{q}\c_{p},\)
      \item \(\sigma_{a}e(\nu) = e(s_{a}(\nu))\sigma_{a},~\sigma_{a}\c_{p} = \c_{s_{a}(p)}\sigma_{a},\)
      \item \(\sigma_{a}x_{p}e(\nu) = x_{p}\sigma_{a}e(\nu)~(p \neq a,a+1),\)
      \item \(\sigma_{a}x_{a+1} - x_{a}\sigma_{a} = \begin{cases}
            e(\nu) & \text{if $\nu_a = \nu_{a+1}$}, \\
            -\c_{a}\c_{a+1}e(\nu) & \text{if $\nu_a = c(\nu_{a+1})$}, \\
            0 & \text{otherwise,}
         \end{cases}\)
      \item \(x_{a+1}\sigma_{a} - \sigma_{a}x_{a} = \begin{cases}
            e(\nu) & \text{if $\nu_a = \nu_{a+1}$}, \\
            \c_{a}\c_{a+1}e(\nu) & \text{if $\nu_a = c(\nu_{a+1})$}, \\
            0 & \text{otherwise,}
         \end{cases}\)
      \item \(\sigma_{a}^2e(\nu) = \tilde{Q}_{\nu_{a},\nu_{a+1}}(x_a,x_{a+1})e(\nu),\)
      \item \(\sigma_{a}\sigma_{b} = \sigma_{b}\sigma_{a}~(|a-b|>1),\)
      \item \begin{align*}
         &(\sigma_{a+1}\sigma_{a}\sigma_{a+1} - \sigma_{a}\sigma_{a+1}\sigma_{a})e(\nu) \\
         &= \begin{cases}
            \frac{\tilde{Q}_{\nu_{a},\nu_{a+1}}(x_{a+2},x_{a+1}) - \tilde{Q}_{\nu_{a},\nu_{a+1}}(x_{a},x_{a+1})}{x_{a+2} - x_a} & \text{if $\nu_a = \nu_{a+2}$}, \\
            \frac{\tilde{Q}_{\nu_{a},\nu_{a+1}}(x_{a+2},x_{a+1}) - \tilde{Q}_{\nu_{a},\nu_{a+1}}(-x_{a},x_{a+1})}{x_{a+2} + x_a} & \text{if $\nu_a = c(\nu_{a+2})$}, \\
            0 & \text{otherwise.}
         \end{cases}
      \end{align*}
   \end{enumerate}
   Here, for \((i,\varepsilon) \in J\),
   \[c(i,\varepsilon) = \begin{cases} (i,1 - \varepsilon) & \text{if } i \in I_{\mathsf{even}}, \\ (i,\varepsilon) & \text{if } i \in I_{\mathsf{odd}} \end{cases}\]
   denotes the involution on \(J\), and, for \(\nu \in J^{\beta}\) and \(1 \leq i \leq n\),
   \[c_{i}(\nu) = (\nu_1,\ldots,\nu_{i-1},c(\nu_i),\nu_{i+1},\ldots,\nu_n).\]
   We further set
   \begin{align*}
      J^c &\coloneqq \{(i,\varepsilon) \in J \mid c(i,\varepsilon) = (i,\varepsilon)\} \\
      &= \{(i,\varepsilon) \in J \mid i \in I_{\mathsf{odd}} \}.
   \end{align*}
\end{dfn}
The relations of \(RC_{\beta}\) are homogeneous with respect to the \(\mathbb{Z}\)-grading
\begin{align*}
      \deg(e(\nu)) &= 0, \\
      \deg(x_{i}e(\nu)) &= 2, \\
      \deg(\sigma_{k}e(\nu)) &= -a_{\pi_{1}(\nu_k) \pi_{1}(\nu_{k+1})}, \\
      \deg(\c_i) &= 0,
\end{align*}
and \(RC_{\beta}\) thereby becomes a \(\mathbb{Z}\)-graded algebra.
The superalgebra \(RC_{\beta}\) further admits an antiautomorphism \(\psi \colon RC_{\beta} \to RC_{\beta}^{\mathsf{op}}\) given by
   \[\begin{array}{ccc}
      RC_{\beta} & \stackrel{\sim}{\longrightarrow} & RC_{\beta}^{\mathsf{op}} \\
      \rotatebox{90}{$\in$} & & \rotatebox{90}{$\in$} \\
      e(\nu) & \longmapsto & e(\nu) \\
      x_i & \longmapsto & x_i \\
      \sigma_{k} & \longmapsto & \sigma_k \\
      \c_i & \longmapsto & -\sqrt{-1} \c_i
   \end{array}\]
   where \(RC_{\beta}^{\mathsf{op}}\) denotes the superalgebra on the same generators as \(RC_{\beta}\), with the relation (RC2) involving \(\c_i\) replaced by
   \[\c_p \c_q + \c_q \c_p = -2\delta_{p,q}.\]
\begin{rem}[{\cite[Corollary 3.9]{KKT}}]\label{rem:PBW for quiver Hecke--Clifford}
   For each \(w \in \mathfrak{S}_n\), fix a reduced expression \(w = s_{i_1} \cdots s_{i_p}\). Then the set
   \[\{x_{1}^{a_1} \cdots x_{n}^{a_{n}}\c_{1}^{\eta_1} \cdots \c_{n}^{\eta_n}\sigma_{w}e(\nu) \mid a_i \in \mathbb{Z}_{\geq 0},~\eta_i \in \mathbb{Z}/2\mathbb{Z},~w \in \mathfrak{S}_n, \nu \in J^{\beta}\}\]
   forms a basis of \(RC_{\beta}\).
\end{rem}

   Let \(RC_{\beta} \sgMod\) denote the (non-full) subcategory of \(RC_{\beta} \sMod\) consisting of \(\mathbb{Z}\)-graded supermodules whose \(\mathbb{Z}\)-grading is compatible with the \(\mathbb{Z}/2\mathbb{Z}\)-grading, together with all morphisms that preserve the \(\mathbb{Z}\)-grading.
   The full subcategory of \(RC_{\beta} \sgMod\) consisting of finite-dimensional modules is denoted by \(RC_{\beta} \sgmod\).

   Compatibility of the \(\mathbb{Z}\)-grading and the \(\mathbb{Z}/2\mathbb{Z}\)-grading means that
   \[M = M_{\overline{0}} \oplus M_{\overline{1}} = \bigoplus_{i \in \mathbb{Z}} M_i\]
   is a decomposition relative to each grading satisfying
   \[M_{\overline{p}} = \bigoplus_{i \in \mathbb{Z}} M_i \cap M_{\overline{p}}~(p=0,1).\]

   For \(M \in RC_{\beta} \sgMod\), write \(qM\) for the grading shift of \(M\) by \(1\), defined by
   \[(qM)_n = M_{n-1}~(n \in \mathbb{Z}).\]
   The shift \(q.-\) defines an auto-equivalence on \(RC_{\beta} \sgMod\) and \(RC_{\beta} \sgmod\).
   The space of degree-\(r\) morphisms between \(M,N \in  RC_{\beta} \sgMod\) is then given by
   \[\Hom_{RC_{\beta}}(q^{-r}M,N) = \Hom_{RC_{\beta}}(M,q^{r}N),\]
   and we set
   \[\HOM_{RC_{\beta}}(M,N) \coloneqq \bigoplus_{r \in \mathbb{Z}} \Hom_{RC_{\beta}}(M,q^{r}N).\]

\begin{dfn}\label{dfn:convolution product}
   Let \(\beta_1,\beta_2 \in \mathsf{Q}^{+}\) and set \(| \beta_i | = n_i~(i = 1,2)\).
   The idempotent \(e(\beta_1,\beta_2) \in RC_{\beta_1 + \beta_2}\) is defined by
   \[e(\beta_1,\beta_2) = \sum_{\nu \in J^{\beta_1 + \beta_2}, (\nu_1,\ldots,\nu_{n_1}) \in J^{\beta_1}}e(\nu).\]
   There is then an algebra homomorphism
   \[\begin{array}{ccc}
      RC_{\beta_1} \otimes RC_{\beta_2} & \stackrel{\sim}{\longrightarrow} & e(\beta_1,\beta_2)RC_{\beta_1 + \beta_2}e(\beta_1,\beta_2) \\
      \rotatebox{90}{$\in$} & & \rotatebox{90}{$\in$} \\
      e(\nu) \otimes e(\eta) & \longmapsto & e(\nu * \eta) \\
      x_i \otimes 1 & \longmapsto & x_i e(\beta_1,\beta_2) \\
      1 \otimes x_i & \longmapsto & x_{n_1 + i} e(\beta_1,\beta_2) \\
      \sigma_{k} \otimes 1 & \longmapsto & \sigma_k e(\beta_1,\beta_2) \\
      1 \otimes \sigma_k & \longmapsto & \sigma_{n_1 + k} e(\beta_1,\beta_2) \\
      \c_i \otimes 1 & \longmapsto & \c_i e(\beta_1,\beta_2) \\
      1 \otimes \c_i & \longmapsto & \c_{n_1 + i} e(\beta_1,\beta_2)
   \end{array}\]
   where \(\nu * \eta = (\nu_1,\ldots,\nu_{n_1},\eta_1,\ldots,\eta_{n_2}) \in J^{\beta_1 + \beta_2}\).

   For \(M_i \in RC_{\beta_i} \sgMod\), the \textit{convolution product} (or \textit{induced module}) \(M_1 \circ M_2 \in RC_{\beta_1 + \beta_2} \sgMod\) is defined by
   \begin{align*}
      M_1 \circ M_2 &= \Ind^{RC_{\beta_1 + \beta_2}}_{RC_{\beta_1} \otimes RC_{\beta_2}} M_{1} \otimes M_{2}\\
      &= RC_{\beta_1 + \beta_2}e(\beta_1,\beta_2) \otimes_{(RC_{\beta_1} \otimes RC_{\beta_2})}(M_1 \otimes_{\Bbbk} M_2).
   \end{align*}
   The \textit{coinduced module} \(\Coind^{RC_{\beta_1 + \beta_2}}_{RC_{\beta_1,\beta_2}} M_1 \otimes M_2\) is defined analogously by
   \[\Coind^{RC_{\beta_1 + \beta_2}}_{RC_{\beta_1,\beta_2}} M_{1} \otimes M_{2} = \HOM_{RC_{\beta_1,\beta_2}}(e(\beta_1,\beta_2)RC_{\beta_1 + \beta_2},M_1 \otimes M_2),\]
   on which an element \(a \in RC_{\beta_1 + \beta_2}\) acts via
   \[(a.f)(b) = (-1)^{\Par(a)(\Par(f) + \Par(b))}f(b.a).\]
\end{dfn}

\begin{rem}\label{rem:convolution product of finite dimensional module is finite dimensional}
   By Remark \ref{rem:PBW for quiver Hecke--Clifford},
   \[RC_{\beta_1 + \beta_2} = \bigoplus_{w \in \mathfrak{S}_{n_1,n_2}} \sigma_{w}.RC_{\beta_1,\beta_2},\]
   which yields
   \begin{align*}
      M_1 \circ M_2 &=RC_{\beta_1 + \beta_2}e(\beta_1,\beta_2) \otimes_{RC_{\beta_1} \otimes RC_{\beta_2}} M_1 \otimes M_2 \\
      &\simeq \bigoplus_{w \in \mathfrak{S}_{n_1,n_2}} \sigma_{w}.(M_1 \otimes M_2).
   \end{align*}
   In particular, if \(M_1\) and \(M_2\) are finite-dimensional, then so is \(M_1 \circ M_2\).
   Recall that \(\mathfrak{S}_{n_1,n_2} \subset \mathfrak{S}_{n_1 + n_2}\) is defined by
   \[\mathfrak{S}_{n_1,n_2} = \{w \in \mathfrak{S}_{n_1+n_2} \mid w(k) < w(k+1) (k \neq n_1)\}.\]
   In other words, \(\mathfrak{S}_{n_1,n_2}\) is the set of minimal length coset representatives of \(\mathfrak{S}_{n_1+n_2}/\mathfrak{S}_{n_1} \times \mathfrak{S}_{n_2}\).
   Let \(w[n_1,n_2]\) denote the unique longest element in \(\mathfrak{S}_{n_1,n_2}\).
   Likewise, \(\Coind_{\beta_1,\beta_2}^{\beta_1 + \beta_2} M_1 \otimes M_2\) is finite-dimensional.
\end{rem}
   For \(M \in RC_{\beta_1 + \beta_2} \sgMod\), the \textit{restricted module} \(\Res_{\beta_1,\beta_2}^{\beta_1 + \beta_2} M \in RC_{\beta_1} \otimes RC_{\beta_2} \sgMod\) is defined by
   \[\Res_{\beta_1,\beta_2}^{\beta_1 + \beta_2}M \coloneqq e(\beta_1,\beta_2)M.\]

Frobenius reciprocity holds for these functors.
As this is an analogue of a well-known result, we omit the proof.

\begin{prop}[cf. {\cite[(2.2)]{LV}}]\label{prop:hom-tensor duality for quiver Hecke--Clifford superalgebra}
   Let \(L \in RC_{\beta_1} \sgMod\), \(M \in RC_{\beta_2} \sgMod\), and \(N \in RC_{\beta_1 + \beta_2} \sgMod\). Then the following natural isomorphisms of graded modules hold:
   \begin{align*}
      \HOM_{RC_{\beta_1 + \beta_2}}(\Ind_{\beta_1,\beta_2}^{\beta_1 + \beta_2}L \otimes M,N) &\simeq \HOM_{RC_{\beta_1} \otimes RC_{\beta_2}}(L \otimes M,\Res_{\beta_1,\beta_2}^{\beta_1 + \beta_2}N), \\
      \HOM_{RC_{\beta_1} \otimes RC_{\beta_2}}(\Res_{\beta_1,\beta_2}^{\beta_1 + \beta_2}N,L \otimes M) &\simeq \HOM_{RC_{\beta_1 + \beta_2}}(N,\Coind_{\beta_1,\beta_2}^{\beta_1 + \beta_2}L \otimes M).
   \end{align*}
\end{prop}

\subsection{Graded dual modules and convolution product}\label{subsec:graded dual}

We introduce (co)induced modules and dual modules, and discuss their interrelations (cf. \cite{LV}).

   Let \(RC_{\beta} \sgMod_{\mathsf{L}}\) denote the full subcategory of \(RC_{\beta} \sgMod\) consisting of modules \(M \in RC_{\beta} \sgMod\) that satisfy the following two conditions:
   \begin{enumerate}[(i)]
      \item each \(M_i\) is finite-dimensional;
      \item \(M_k = 0\) for sufficiently small \(k \in \mathbb{Z}\).
   \end{enumerate}
   For \(M \in RC_{\beta} \sgMod_{\mathsf{L}}\), the \textit{graded dual} \(M^*\) of \(M\) is defined by
   \[M^* = \bigoplus_{i \in \mathbb{Z}}\Hom_{\Bbbk}(M_i,\Bbbk).\]
   The space \(M^*\) carries a natural \((\mathbb{Z}/2\mathbb{Z},\mathbb{Z})\)-grading whose \(\mathbb{Z}\)-component reads
   \[(M^{*})_i = \Hom_{\Bbbk}(M_{-i},\Bbbk).\]
   We endow \(M^*\) with a right \(RC_{\beta}\)-module structure via
   \[(f.a)(m) \coloneqq f(a.m)~(f \in M^{*}, a \in RC_{\beta}, m \in M_i),\]
   and twist this action by the anti-involution of \(RC_{\beta}\) to obtain \(M^* \in RC_{\beta} \sgMod\).

\begin{rem}
   The graded dual \(M^{*} \in RC_{\beta} \sgMod\) satisfies the following two conditions:
   \begin{enumerate}
      \item[(i)] each \((M^{*})_i\) is a finite-dimensional super vector space;
      \item[(ii)$'$] \((M^{*})_{k} = 0\) for sufficiently large \(k \in \mathbb{Z}\).
   \end{enumerate}
   Let \(RC_{\beta} \sgMod_{\mathsf{L}'}\) denote the full subcategory of \(RC_{\beta} \sgMod\) consisting of objects satisfying these two conditions.
   For \(N \in RC_{\beta} \sgMod_{\mathsf{L}'}\), the graded dual \(N^{*}\) is defined analogously and satisfies \(N^* \in RC_{\beta} \sgMod_{\mathsf{L}}\).
   For \(M \in RC_{\beta} \sgMod_{\mathsf{L}}\), there is a natural isomorphism \((M^*)^* \simeq M\).
   Hence the functor
   \[(-)^{*} \colon RC_{\beta} \sgMod_{\mathsf{L}} \to RC_{\beta} \sgMod_{\mathsf{L}'}\]
   is a contravariant equivalence.
   Since \(RC_{\beta} \sgmod\) is a full subcategory of \(RC_{\beta} \sgMod_{\mathsf{L}} \cap RC_{\beta} \sgMod_{\mathsf{L}'}\) and is stable under \((-)^{*}\), the graded dual \((-)^{*}\) induces a contravariant auto-equivalence on \(RC_{\beta} \sgmod\).
\end{rem}

\begin{lem}[cf. {\cite[Theorem 2.2]{LV},\cite[Theorem 6.2]{M2}}]\label{lem:ind and coind are isom}
   For \(M_i \in RC_{\beta_i} \sgMod_{\mathsf{L}}~(i=1,2)\), there is an isomorphism
   \[\Ind_{\beta_1,\beta_2}^{\beta_1 + \beta_2} M_1 \otimes M_2 \simeq q^{-(\beta_1,\beta_2)}\Coind_{\beta_2,\beta_1}^{\beta_2 + \beta_1} M_2 \otimes M_1.\]
\end{lem}
\begin{proof}
   In view of Proposition \ref{prop:hom-tensor duality for quiver Hecke--Clifford superalgebra}, we first construct a morphism of \(RC_{\beta_1} \otimes RC_{\beta_2}\)-modules
   \[F \in \HOM_{RC_{\beta_1} \otimes RC_{\beta_2}} (M_1 \otimes M_2 , \Coind_{\beta_2,\beta_1}^{\beta_2 + \beta_1} M_2 \otimes M_1)\]
   compatible with the decomposition
   \[e(\beta_2,\beta_1)RC_{\beta_2 + \beta_1} = \bigoplus_{w \in \mathfrak{S}_{n_2} \times \mathfrak{S}_{n_1} \backslash \mathfrak{S}_{n_2 + n_1}}RC_{\beta_2,\beta_1}\sigma_{w}.\]
   The map \(F\) is given by
   \[\begin{array}{ccc}
      M_1 \otimes M_2 & \stackrel{F}{\longrightarrow} & \HOM_{RC_{\beta_2} \otimes RC_{\beta_1}}(e(\beta_2,\beta_1)RC_{\beta_1 + \beta_2}, M_2 \otimes M_1) \\
      \rotatebox{90}{$\in$} & & \rotatebox{90}{$\in$} \\
      u_1 \otimes u_2 & \longmapsto & F_{u_1 \otimes u_2} \colon \sigma_{w} \longmapsto \delta_{w,w[n_1,n_2]} (-1)^{\Par(u_1)\Par(u_2)}u_2 \otimes u_1
   \end{array}\]
   For a general element \(h \sigma_{w} \in e(\beta_1, \beta_2)RC_{\beta_1 + \beta_2}\),
   \[F_{u_1 \otimes u_2}(h \sigma_{w}) = (-1)^{\Par(h)(\Par(u_1) + \Par(u_2))}h F_{u_1 \otimes u_2}(\sigma_{w}).\]
   A direct computation yields
   \begin{align*}
      \deg(F) &= \deg(F_{u_1 \otimes u_2}) - \deg(u_1 \otimes u_2) \\
      &= (\deg(u_1 \otimes u_2) - \deg(\sigma_{w[n_1,n_2]})) - \deg(u_1 \otimes u_2) \\
      &= - \deg(\sigma_{w[n_1,n_2]}) \\
      &= -(\beta_1,\beta_2).
   \end{align*}
   By Frobenius reciprocity (Proposition \ref{prop:hom-tensor duality for quiver Hecke--Clifford superalgebra}), \(F\) yields a morphism of graded \(RC_{\beta_1 + \beta_2}\)-modules
   \[\mathcal{F} \colon \Ind_{\beta_1,\beta_2}^{\beta_1 + \beta_2} M_1 \otimes M_2 \to q^{-(\beta_1,\beta_2)}\Coind_{\beta_2,\beta_1}^{\beta_2 + \beta_1} M_2 \otimes M_1.\]
   Since each \(\mathbb{Z}\)-homogeneous component of \(M_1\) and \(M_2\) is finite-dimensional, the corresponding components on both sides are also finite-dimensional and of equal dimension.
   By construction, \(\mathcal{F}\) is injective. Indeed, if \(\mathcal{F}(u) = 0\), then
   \[\mathcal{F}(u)(\sigma_{w[n_1,n_2]}) = u = 0.\]
   Therefore \(\mathcal{F}\) is an isomorphism.
\end{proof}

\begin{cor}\label{cor:dual and convolution product}
   For \(M_i \in RC_{\beta_i} \sgMod_{\mathsf{L}}~(i=1,2)\), there is an isomorphism
   \[(M_1 \circ M_2)^{*} \simeq q^{(\beta_1,\beta_2)}M_{2}^{*} \circ M_{1}^{*}.\]
\end{cor}
\begin{proof}
   For \(N \in RC_{\beta_1 + \beta_2} \sgMod_{\mathsf{L}'}\), we have the following chain of natural isomorphisms of graded modules:
   \begin{align*}
      \HOM_{RC_{\beta_1 + \beta_2}}((M_1 \circ M_2)^*, N) &\simeq \HOM_{RC_{\beta_1 + \beta_2}}(M_1^* \circ M_2^*, N) \\
      &\simeq \HOM_{RC_{\beta_1} \otimes RC_{\beta_2}}(M_1^* \otimes M_2^*, \Res_{\beta_1,\beta_2}^{\beta_1 + \beta_2} N) \\
      &\simeq \HOM_{RC_{\beta_1} \otimes RC_{\beta_2}}(\Res_{\beta_1,\beta_2}^{\beta_1 + \beta_2} N^*,(M_1^* \otimes M_2^*)^*) \\
      &\simeq \HOM_{RC_{\beta_1 + \beta_2}}(N^*, \Coind^{\beta_1 + \beta_2}_{\beta_1,\beta_2} (M_1^* \otimes M_2^*)^*) \\
      &\simeq \HOM_{RC_{\beta_1 + \beta_2}}(\Coind^{\beta_1 + \beta_2}_{\beta_1,\beta_2} M_1^* \otimes M_2^*,N) \\
      &\simeq \HOM_{RC_{\beta_1 + \beta_2}}(q^{(\beta_1,\beta_2)}\Ind^{\beta_2 + \beta_1}_{\beta_2,\beta_1} M_2^* \otimes M_1^*,N).
   \end{align*}
   Restricting to degree-zero morphisms yields
   \[\Hom_{RC_{\beta_1 + \beta_2}}((M_1 \circ M_2)^*, N) \simeq \Hom_{RC_{\beta_1 + \beta_2}}(q^{(\beta_1,\beta_2)}\Ind^{\beta_2 + \beta_1}_{\beta_2,\beta_1} M_2^* \otimes M_1^*,N).\]
   The claim \((M_1 \circ M_2)^* \simeq q^{(\beta_1,\beta_2)}M_2^* \circ M_1^*\) then follows from Yoneda's lemma.
\end{proof}

\begin{rem}
   These results play an important role in the study of the structure of tensor products of irreducible representations over symmetric quiver Hecke(--Clifford super)algebras, as will be discussed in Section \ref{subsec:symmetric quiver Hecke--Clifford superalgebras}.
\end{rem}

\subsection{Intertwiners and R-matrices}\label{subsec:intertwiners}

We recall the construction of the ``polynomial module'' for quiver Hecke--Clifford superalgebras and the R-matrix defined in terms of it.
The polynomial module plays an essential role also in the construction of Section \ref{sec:schur--weyl duality via quiver hecke--clifford superalgebra} (cf. \cite{KKK}).

Suppose \(|\beta| = n\).
Let \(\mathcal{A}_{\beta}\) denote the sub-superalgebra of \(RC_{\beta}\) generated by
\[\mathcal{A}_{\beta} = \langle e(\nu), x_i, \c_i \mid \nu \in J^{\beta},  1 \leq i \leq n \rangle_{\Bbbk \text{-subalg}},\]
and let \(\mathcal{K}_{\beta}\) be its localization
\[\mathcal{K}_{\beta} = \mathcal{A}_{\beta}[(x_a \pm x_b)^{-1} \mid 1 \leq a < b \leq n].\]
For \(1 \leq a, b \leq n\) with \(a \neq b\), the central element \(R_{a,b}\) of \(\mathcal{K}_{n}\) is defined by
\begin{align*}
   R_{a,b} = \sum_{\pi_{1}(\nu_a) \neq \pi_{1}(\nu_b)} \tilde{Q}_{\nu_a,\nu_b}(x_a,x_b)e(\nu) - \sum_{\nu_a = \nu_b} \frac{1}{(x_a - x_b)^2}e(\nu) - \sum_{\nu_a = c\nu_b} \frac{1}{(x_a + x_b)^2}e(\nu).
\end{align*}
Let \(\mathcal{KS}_{\beta}\) be the superalgebra generated by \(\mathcal{K}_{\beta}\) and the even elements \(\tilde{s}_1, \ldots, \tilde{s}_{n-1}\), subject to:
\begin{enumerate}[(i)]
   \item the relations of \(\mathcal{K}_{\beta}\),
   \item the braid relations for \(\tilde{s}_k\),
   \item \(\tilde{s}^2_k = R_{k,k+1}~(1 \leq k \leq n - 1),\)
   \item \(\tilde{s}_{k}a = (s_{k}.a)\tilde{s}_k~(a \in \mathcal{K}_{\beta}).\)
\end{enumerate}
Here, \(s_k = (k,k + 1) \in \mathfrak{S}_n\) is the adjacent transposition, and \(s_k.a\) denotes the action of \(\mathfrak{S}_n\) on \(a \in \mathcal{K}_{\beta}\) by permutation.
The superalgebra \(\mathcal{KS}_{\beta}\) admits an antiautomorphism \(\psi \colon \mathcal{KS}_{\beta} \to \mathcal{KS}_{\beta}^{\mathsf{op}}\) given by
   \[\begin{array}{ccc}
      \mathcal{KS}_{\beta} & \stackrel{\sim}{\longrightarrow} & \mathcal{KS}_{\beta}^{\mathsf{op}} \\
      \rotatebox{90}{$\in$} & & \rotatebox{90}{$\in$} \\
      e(\nu) & \longmapsto & e(\nu) \\
      x_i & \longmapsto & x_i \\
      \tilde{s}_k & \longmapsto & \tilde{s}_k \\
      \c_i & \longmapsto & -\sqrt{-1} \c_i
   \end{array}\]
   where \(\mathcal{KS}_{\beta}^{\mathsf{op}}\) denotes the superalgebra on the same generators as \(\mathcal{KS}_{\beta}\), with the relation \(\c_i \c_j + \c_j \c_i = 2\delta_{ij}\) replaced by
   \[\c_p \c_q + \c_q \c_p = -2\delta_{p,q}.\]

\begin{thm}[{\cite[Theorem 3.8.]{KKT}}]
Let \(|\beta| = n\).
Consider the following assignment on the generators of \(RC_{\beta}\):
   \[
   \begin{array}{rccc}
      \iota \colon & RC_{\beta} & \longrightarrow & \mathcal{KS}_{\beta} \\
      & e(\nu) & \longmapsto & e(\nu) \\
      & x_{i} & \longmapsto & x_{i} \\
      & \c_{i} & \longmapsto & \c_{i} \\
      & \sigma_{a}e(\nu) & \longmapsto & \iota(\sigma_{a}e(\nu))
   \end{array}\]
   where
   \[\iota(\sigma_{a}e(\nu)) = \begin{cases}
      \tilde{s}_{a}e(\nu) & \text{if}~\pi_{1}(\nu_{a}) \neq \pi_{1}(\nu_{a+1}), \\
      (\tilde{s}_{a} - (x_{a} - x_{a+1})^{-1})e(\nu) & \text{if}~\nu_{a} = \nu_{a+1} \notin J^c, \\
      (\tilde{s}_{a} + (x_{a} + x_{a+1})^{-1}\c_{a}\c_{a+1})e(\nu) & \text{if}~\nu_{a} = c(\nu_{a+1}) \notin J^c, \\
      (\tilde{s}_{a} - (x_{a} - x_{a+1})^{-1} + (x_{a} + x_{a+1})^{-1}\c_{a}\c_{a+1})e(\nu) & \text{if}~\nu_{a} = \nu_{a+1} \in J^c.
   \end{cases}\]
   This assignment uniquely extends to an injective homomorphism of superalgebras \(\iota \colon RC_{\beta} \to \mathcal{KS}_{\beta}\).
\end{thm}

Using this, we define the intertwiner as an element of \(RC_{\beta}\) that maps to a polynomial multiple of \(\tilde{s}_a\) under \(\iota\).

\begin{dfn}\label{dfn:intertwiner}
   Let \(|\beta| = n\).
   For \(1 \leq a \leq n - 1\), the \textit{intertwiner} \(\varphi_{a} \in RC_{\beta}\) is defined by
   \[\varphi_{a}e(\nu) \coloneqq \begin{cases}
      \sigma_a e(\nu) & \text{if}~\pi_{1}(\nu_a) \neq \pi_{1}(\nu_{a+1}), \\
      (-1)^{\pi_2(\nu_{a})}(\sigma_{a}x_{a} - x_{a}\sigma_{a})e(\nu) & \text{if}~\nu_a = \nu_{a+1} \notin J^c, \\
      (-1)^{\pi_2(\nu_{a})}(\sigma_{a}x_{a} + x_{a}\sigma_{a})e(\nu) & \text{if}~\nu_a = c(\nu_{a+1}) \notin J^c, \\
      (\sigma_{a}x_{a}^{2}- x_{a}^{2}\sigma_{a})e(\nu) & \text{if}~\nu_a = \nu_{a+1} \in J^c.
   \end{cases}\]
\end{dfn}

By construction, \(\iota(\varphi_{a}e(\nu))\) takes the form
   \[\iota(\varphi_{a}e(\nu)) \coloneqq \begin{cases}
      \tilde{s}_{a} e(\nu) & \text{if}~\pi_{1}(\nu_a) \neq \pi_{1}(\nu_{a+1}), \\
      (-1)^{\pi_2(\nu_a)}(x_{a} - x_{a+1})\tilde{s}_{a}e(\nu),  & \text{if}~\nu_a = \nu_{a+1} \notin J^c, \\
      (-1)^{\pi_2(\nu_{a})}(x_{a} + x_{a+1})\tilde{s}_{a}e(\nu) & \text{if}~\nu_a = c(\nu_{a+1}) \notin J^c, \\
      (x_{a+1}^{2} - x_{a}^2)\tilde{s}_{a}e(\nu) & \text{if}~\nu_a = \nu_{a+1} \in J^c.
   \end{cases}\]
   Observe that \(\varphi_{a}e(\nu)\) is even for all \(1 \leq a \leq n\) and \(\nu \in J^{\beta}\).
   This construction extends that of the intertwiner in \cite{KKK}.
   The following is an analogue of \cite[Lemma 1.5]{KKK}.

\begin{lem}\label{lem:property of intertwiner}
   For \(1 \leq a \leq n-1\), let \(\varphi_{a} \in RC_{\beta}\) be defined as above.
   The following hold:
   \begin{enumerate}[(i)]
      \item \(\varphi_a^2 e(\nu) = \begin{cases}
         \tilde{Q}_{\nu_a,\nu_{a+1}}(x_a,x_{a+1})e(\nu) & \text{if } \pi_1(\nu_a) \neq \pi_1(\nu_{a+1}), \\
         e(\nu) & \text{if } \pi_1(\nu_a) = \pi_1(\nu_{a+1}) \in I_{\mathsf{even}}, \\
         2(x_a^2 + x_{a+1}^2)e(\nu) & \text{if } \pi_1(\nu_a) = \pi_1(\nu_{a+1}) \in I_{\mathsf{odd}}.
      \end{cases}\)
      \item The family \(\{\varphi_a\}_{1 \leq a < n}\) satisfies the braid relations.
      \item For each \(w \in \mathfrak{S}_n\), choose a reduced expression \(w = s_{a_1} \cdots s_{a_l}\) and set \(\varphi_w = \varphi_{a_1} \cdots \varphi_{a_l}\). Then \(\varphi_w\) is independent of the choice of reduced expression.
      \item For \(w \in \mathfrak{S}_n\) and \(1 \leq k \leq n\), we have \(\varphi_{w}x_k = x_{w(k)}\varphi_{w}\) and \(\varphi_{w}\c_k = \c_{w(k)}\varphi_{w}\).
      \item If \(w \in \mathfrak{S}_n\) and \(1 \leq k \leq n\) satisfy \(w(k+1) = w(k) + 1\), then \(\varphi_{w}\sigma_k = \sigma_{w(k)}\varphi_{w}\).
      \item For \(w \in \mathfrak{S}_n\),
      \begin{align*}
         \varphi_{w^{-1}}\varphi_{w}e(\nu) = \prod_{\substack{a<b, \\ w(a) > w(b)}} &(\tilde{Q}_{\nu_a,\nu_b}(x_a,x_b) + \delta(\pi_1(\nu_a) = \pi_1(\nu_{b}) \in I_{\mathsf{even}}) \\
         &+ 2(x_a^2 + x_b^2)\delta(\pi_1(\nu_a) = \pi_1(\nu_{b}) \in I_{\mathsf{odd}}))e(\nu).
      \end{align*}
   \end{enumerate}
\end{lem}

\begin{proof}
   The proof follows that of \cite[Lemma 1.5]{KKK}.
   Namely, we work in \(\mathcal{KS}_{\beta}\) via the injective map \(\iota\).
   For example, the case \(\pi_1(\nu_a) = \pi_1(\nu_{a+1}) \in I_{\mathsf{odd}}\) in (i) follows from the computation:
   \begin{align*}
      \iota(\varphi_a^2 e(\nu)) &= \iota(\varphi_a e(s_a(\nu)) \varphi_a e(\nu)) \\
      &= (x_{a+1}^2 - x_a^2) \tilde{s}_a (x_{a+1}^2 - x_a^2) \tilde{s}_a e(\nu) \\
      &= -(x_{a}^2 - x_{a+1}^2)^2 (- \frac{1}{(x_a - x_{a+1})^2} - \frac{1}{(x_a + x_{a+1})^2})e(\nu) \\
      &= ((x_a + x_{a+1})^2 + (x_a - x_{a+1})^2)e(\nu) \\
      &= 2(x_a^2 + x_{a+1}^2)e(\nu) = \iota(2(x_a^2 + x_{a+1}^2)e(\nu)).
   \end{align*}
   The proofs of (ii)--(v) are parallel to those of \cite[Lemma 1.5]{KKK}.
   Finally, (vi) follows from (i) by induction on the length of \(w\).
\end{proof}

For \(M \in RC_{\beta_1} \sgMod\) and \(N \in RC_{\beta_2} \sgMod\), the map
\[M \otimes N \to \Res^{\beta_1 + \beta_2}_{\beta_1,\beta_2}N \circ M; \quad u \otimes v \longmapsto (-1)^{\Par(u)\Par(v)} \varphi_{w[n,m]}v \otimes u\]
is a morphism of \(RC_{\beta_1} \otimes RC_{\beta_2}\)-supermodules.
Frobenius reciprocity then yields a morphism \(R_{M,N} \colon M \circ N \to N \circ M\) of \(RC_{\beta_1 + \beta_2}\)-modules, called the \textit{R-matrix}.
The family \(\{R_{M,N}\}\) satisfies the Yang--Baxter equation.

\subsection{Affinization}\label{subsec:affinization}

In this section we recall, following \cite{KP}, the theory of affinization for modules over quiver Hecke--Clifford superalgebras.
The results in the rest of this section will not be used in this paper, but we include them since they might be of independent interest.

In analogy with \cite[Definition 2.1]{KP}, we define an element \(\mathfrak{p}_{i,\beta} \in Z(RC_{\beta})\).

\begin{dfn}
   For \(i \in J\) and \(\beta \in \mathsf{Q}^+\), define \(\mathfrak{p}_{i,\beta} \in Z(RC_{\beta})\) by
   \[\mathfrak{p}_{i,\beta} = \sum_{\nu \in J^{\beta}} (\prod_{\pi_1(\nu_a) = i}((-1)^{\pi_2(\nu_a)}x_a)^{1 + \Par(\pi_1(\nu_a))})e(\nu).\]
\end{dfn}

\begin{dfn}
   Let \(\bar{M} \in RC_{\beta} \sgmod\) be a simple module.
   An \textit{affinization} of \(\bar{M}\) consists of \(\mathsf{M} \in RC_{\beta} \sMod\), an injective morphism \(z_{\mathsf{M}} \in \END_{RC_{\beta}}(\mathsf{M})\) of degree \(d_{\mathsf{M}} \in \mathbb{Z}_{> 0}\), and an isomorphism \(\mathsf{M}/z_{\mathsf{M}}\mathsf{M} \simeq \bar{M}\), subject to the following conditions:
   \begin{enumerate}[(i)]
      \item \(\mathsf{M}\) is a finitely generated free \(\mathbf{k}[z_{\mathsf{M}}]\)-module;
      \item \(\mathfrak{p}_{i} \mathsf{M} \neq 0\) for every \(i \in J\).
   \end{enumerate}
   If, in addition,
   \begin{enumerate}
      \item[(iii)] the exact sequence
      \[\xymatrix{0 \ar[r] & z_{\mathsf{M}}\mathsf{M}/z_{\mathsf{M}}^{2}\mathsf{M} \ar[r] & \mathsf{M}/z_{\mathsf{M}}^{2}\mathsf{M} \ar[r] & \mathsf{M}/z_{\mathsf{M}}\mathsf{M} \ar[r] & 0}\]
      does not split,
   \end{enumerate}
   then \(\mathsf{M}\) is called a \textit{strong affinization}.
   If, moreover, \(d_{M} \in 2\mathbb{Z}_{>0}\), then \(\mathsf{M}\) is called an \textit{even affinization}.
\end{dfn}

Let \(\mathsf{M}\) be a weak affinization of \(\bar{M}\), and let \(\gamma \in \mathsf{Q}^{+}\) and \(0 \neq N \in RC_{\gamma} \sgmod\).
By definition of the affinization, there exists a maximal integer \(s_{\mathsf{M},N} \in \mathbb{Z}_{\geq 0}\) such that \(R_{\mathsf{M},N}(\mathsf{M} \circ N) \subset z_{\mathsf{M}}^s N \circ \mathsf{M}\).
The morphism of \(RC_{\beta + \gamma}\)-modules \(R^{\mathrm{norm}}_{\mathsf{M},N}\) defined by
\[R^{\mathrm{norm}}_{\mathsf{M},N} = z^{-s_{\mathsf{M},N}}R_{\mathsf{M},N} \colon \mathsf{M} \circ N \to N \circ \mathsf{M}.\]
The induced morphism
\[r_{\bar{M},N} \colon \bar{M} \circ N \to N \circ \bar{M}\]
is called the \textit{renormalized R-matrix}.

\begin{ex}
   Let \(i \in I_{\mathsf{odd}}\).
   The algebra \(RC_{2\alpha_i}\) is called the \textit{NilHecke--Clifford superalgebra} \cite{L}.
   Define the irreducible representation \(L(i) \simeq \mathcal{C}_1\) of \(RC_{\alpha_i}\) by
   \[x_1.\c = 0, \quad \c_1.\c^i = \c^{i+1}~(i = 0,1).\]
   The convolution product \(L(i) \circ L(i)\) is an \(RC_{2\alpha_i}\)-module generated by \(\sigma_1.1 \otimes 1\).
   Define an idempotent \(b_2 \in RC_{2\alpha_i}\) by
   \[b_2 = \frac{\c_1 \c_2 - 1}{2}\sigma_1 x_1.\]
   Indeed,
   \begin{align*}
      b_2^2 &= \frac{\c_1 \c_2 - 1}{2}\sigma_1 x_1 \frac{\c_1 \c_2 - 1}{2}\sigma_1 x_1 \\
      &= \frac{\c_1 \c_2 - 1}{2}(x_2 \sigma_1 - 1 - \c_1 \c_2)\frac{\c_1 \c_2 - 1}{2}\sigma_1 x_1 \\
      &= \frac{\c_1 \c_2 - 1}{2} \frac{(-1)^2 - (\c_1\c_2)^2}{2}\sigma_1 x_1 \\
      &= \frac{\c_1 \c_2 - 1}{2}\sigma_1 x_1 = b_2.
   \end{align*}
   Since the projective module \(RC_{2\alpha_i}b_2\) is isomorphic to the polynomial representation \(RC_{2\alpha_i}/RC_{2\alpha_i}\sigma_1\), there exists a surjection \(RC_{2\alpha_i}b_2 \twoheadrightarrow L(i) \circ L(i)\).
   This provides a strong affinization of \(L(i) \circ L(i)\).
\end{ex}

Using the element \(\mathfrak{p}_{i,\beta}\) and the affinization introduced above, analogous results hold for quiver Hecke--Clifford superalgebras by arguments parallel to those of \cite{KP}.
For example, the following is an analogue of \cite[Theorem 2.13]{KP}; we omit the proof.

\begin{thm}
   Let \((\mathsf{M},z_{\mathsf{M}})\) and \((\mathsf{N},z_{\mathsf{N}})\) be affinizations of simple modules \(\bar{M}\) and \(\bar{N}\), respectively, at least one of which is real.
   Then \(R^{\mathrm{norm}}_{\mathsf{M},\mathsf{N}}|_{z_{\mathsf{M}} = z_{\mathsf{N}} = 0} \colon \bar{M} \circ \bar{N} \to \bar{N} \circ \bar{M}\) does not vanish and equals a scalar multiple of \(r_{\bar{M},\bar{N}}\).
\end{thm}

\subsection{Symmetric quiver Hecke--Clifford superalgebras}\label{subsec:symmetric quiver Hecke--Clifford superalgebras}

We restrict attention to the symmetric case and discuss the representation theory of quiver Hecke--Clifford superalgebras.
In particular, we explain how the corresponding quiver Hecke algebra can be used to study quiver Hecke--Clifford superalgebras.

\begin{dfn}
   A quiver Hecke--Clifford superalgebra \(RC_{\beta}\) is called \textit{symmetric} if:
   \begin{enumerate}[(i)]
      \item the Cartan matrix \(A\) is symmetric;
      \item \((\alpha_i,\alpha_j) = a_{ij}\);
      \item \(\tilde{Q}_{ij}(u,v) \in \Bbbk[(-1)^{\pi_2(i)}u - (-1)^{\pi_2(j)}v]\).
   \end{enumerate}
\end{dfn}

\begin{lem}\label{lem:symmetric quiver Hecke superalgebra is quiver Hecke algebra}
   If \(RC_{\beta}\) is symmetric, then \(I_{\mathsf{odd}} = \varnothing\).
\end{lem}
\begin{proof}
   We may assume that \(A\) is indecomposable.
   Suppose, for a contradiction, that \(I_{\mathsf{odd}} \neq \varnothing\), and pick \(i \in I_{\mathsf{odd}}\).
   By the indecomposability of \(A\), there exists \(j \neq i\) such that \(a_{ij} \in 2\mathbb{Z}_{<0}\).
   It follows that
   \[\tilde{\mathcal{Q}}_{ij}(u,v) - t_{i,j;-a_{ij},0}(u - (-1)^{\pi_2(j)}v)^{-a_{ij}} \in \sum_{r < -a_{ij}} \Bbbk.(u - (-1)^{\pi_2(j)}v)^{r}.\]
   Hence \((-a_{ij} - 1,1) \in S_{ij}\), which contradicts the definition of \(S_{ij}\).
\end{proof}

Throughout the remainder of this paper, we restrict our attention to symmetric quiver Hecke--Clifford superalgebras.

\begin{prop}[{\cite[Section 3.5]{KKT}}]
   For \(\beta \in \mathsf{Q}^{+}\), define the idempotent \(e^{\dagger}(\beta) \in RC_{\beta}\) by
   \[e^{\dagger}(\beta) = \sum_{i , \pi_2(\nu_i) = 0}e(\nu).\]
   Then \(RC_{\beta}e^{\dagger}(\beta)RC_{\beta} = RC_{\beta}\).
   In particular, setting \(RC^{\dagger}_{\beta} = e^{\dagger}(\beta)RC_{\beta}e^{\dagger}(\beta)\), the superalgebras \(RC_{\beta}\) and \(RC^{\dagger}_{\beta}\) are Morita superequivalent.
\end{prop}

\begin{thm}[{\cite[Theorem 3.13]{KKT}}]\label{thm:super morita equivariant for quiver Hecke--Clifford superalgebra}
   Let \(R_{\beta}\) denote the quiver Hecke algebra associated with the same Cartan datum and family of polynomials \(\{\mathcal{Q}_{i,j}\}_{i,j \in I}\) as \(RC_{\beta}\). Then there is an isomorphism of (super)algebras
   \[R_{\beta} \simeq RC^{\dagger}_{\beta}.\]
   In particular, \(R_{\beta}\) and \(RC_{\beta}\) are Morita superequivalent.
\end{thm}
   Explicitly, the isomorphism \(R_{\beta} \simeq RC^{\dagger}_{\beta}\) sends
   \begin{align*}
      \begin{array}{ccc}
         R_{\beta} & \longrightarrow & RC^{\dagger}_{\beta} \\
         \rotatebox{90}{$\in$} & & \rotatebox{90}{$\in$} \\
         e(\nu) & \longmapsto & e(\nu^{\dagger}) \\
         x_{p}e(\nu) & \longmapsto & x_{p}e(\nu^{\dagger}) \\
         \sigma_{a}e(\nu) & \longmapsto & \sigma_{a}e(\nu^{\dagger})
      \end{array}
   \end{align*}
   where \(\nu^{\dagger} = ((\nu_1,\overline{0}),\ldots,(\nu_n,\overline{0})) \in J^{\beta}\).
   Note that \(e^{\dagger}(\beta)\c_i e^{\dagger}(\beta) = 0\).

The category \(R_{\beta} \sgMod\) appearing in the equivalence \(RC^{\dagger}_{\beta} \sgMod \simeq R_{\beta} \sgMod\) admits the following description:
\begin{align*}
   \Ob(R_{\beta} \sgMod ) &= \{M_{\overline{0}} \oplus M_{\overline{1}} \mid M_{\overline{0}},M_{\overline{1}} \in R_{\beta} \gMod \}, \\
   \Hom_{R_{\beta} \sgMod }(M_{\overline{0}} \oplus M_{\overline{1}},N_{\overline{0}} \oplus N_{\overline{1}}) &= \left\{\begin{bmatrix}
   f_{00} & f_{01} \\
   f_{10} & f_{11} \\
   \end{bmatrix}
   \mid f_{ij} \colon M_{\overline{i}} \to N_{\overline{j}} \text{ in }R_{\beta} \gMod  \right\}.
\end{align*}
The equivalence is realized by the pair of functors
\begin{align*}
   e^{\dagger}(\beta).- &\colon RC_{\beta} \sgMod  \rightarrow RC_{\beta}^{\dagger} \sgMod, \\
   RC_{\beta}e^{\dagger}(\beta) \otimes_{RC^{\dagger}_{\beta}} - &\colon RC^{\dagger}_{\beta} \sgMod  \rightarrow RC_{\beta} \sgMod.
\end{align*}
For any \(M \in RC^{\dagger}_{\beta} \sgMod \), there is an isomorphism of super vector spaces
\[
   RC_{\beta}e^{\dagger}(\beta)\otimes_{RC^{\dagger}_{\beta}} M \simeq \mathcal{C}_{n} \otimes_{\Bbbk} M.
\]
In particular, restricting to finite-dimensional modules induces the equivalence
\[RC_{\beta}e^{\dagger}(\beta) \otimes_{RC^{\dagger}_{\beta}} - \colon R_{\beta} \sgmod  \rightarrow RC_{\beta} \sgmod .\]

\begin{prop}\label{prop:morita equivalence for quiver Hecke--Clifford superalgebra}
   Let \(F_{\beta} = e^{\dagger}(\beta).- \colon RC_{\beta} \sgMod \to R_{\beta} \sgMod\) denote the equivalence of categories described above.
   For any \(M_i \in RC_{\beta_i} \sgMod\) (\(i = 1,2\)), there is an isomorphism
   \[F_{\beta_1 + \beta_2}(M_1 \circ M_2) \simeq F_{\beta_1}(M_1) \circ F_{\beta_2}(M_2).\]
   Equivalently, \(F \coloneqq \bigoplus_{\beta \in \mathsf{Q}^+} F_{\beta}\) is a monoidal functor.
\end{prop}
\begin{proof}
   Setting \(|\beta_i| = n_i\) (\(i=1,2\)),
   \begin{align*}
      F_{\beta_1 + \beta_2}(M_1 \circ M_2) &= e^{\dagger}(\beta_1 + \beta_2) RC_{\beta_1 + \beta_2}e(\beta_1,\beta_2) \otimes_{RC_{\beta_1} \otimes RC_{\beta_2}} M_1 \otimes M_2 \\
      &= R_{\beta_1 + \beta_2}. e^{\dagger}(\beta_1 + \beta_2).\mathcal{C}_{n_1 + n_2} .e(\beta_1,\beta_2) \otimes M_1 \otimes M_2 \\
      &= R_{\beta_1 + \beta_2}. e^{\dagger}(\beta_1 + \beta_2).e(\beta_1,\beta_2) \otimes M_1 \otimes M_2 \\
      &= R_{\beta_1 + \beta_2}. e^{\dagger}(\beta_1,\beta_2)\otimes M_1 \otimes M_2 \\
      &= F_{\beta_1}(M_1) \circ F_{\beta_2}(M_2).
   \end{align*}
\end{proof}

We adopt this identification without further comment.
The following lemma is an immediate consequence.

\begin{lem}\label{lem:compatibility of Morita and R-matrix}
   For \(m_i \in M_i \in RC_{\beta_i} \sgMod\) (\(i=1,2\)),
   \begin{align*}
      &F_{\beta_1 + \beta_2}(R_{M_1,M_2})(e^{\dagger}(\beta_1)m_1 \otimes e^{\dagger}(\beta_2)m_2) \\
      &= (-1)^{\Par(m_1)\Par(m_2)}R_{F_{\beta_1}(M_1),F_{\beta_2}(M_2)}(e^{\dagger}(\beta_1)m_1 \otimes e^{\dagger}(\beta_2)m_2).
   \end{align*}
\end{lem}
\begin{proof}
   Setting \(|\beta_i| = n_i\) (\(i=1,2\)),
   \begin{align*}
      &F_{\beta_1 + \beta_2}(R_{M_1,M_2})(m_1 \otimes m_2) \\
      &= e^{\dagger}(\beta_1 + \beta_2)\varphi_{w[n_2,n_1]}.- \circ T_{F_{\beta_1}(M_1),F_{\beta_2}(M_2)} \circ e^{\dagger}(\beta_1 + \beta_2).(m_1 \otimes m_2) \\
      &= e^{\dagger}(\beta_1 + \beta_2)\varphi_{w[n_2,n_1]}.- \circ T_{F_{\beta_1}(M_1),F_{\beta_2}(M_2)}(e^{\dagger}(\beta_1)m_1 \otimes e^{\dagger}(\beta_2)m_2) \\
      &= (-1)^{\Par(m_1)\Par(m_2)}e^{\dagger}(\beta_1 + \beta_2)\varphi_{w[n_2,n_1]}(e^{\dagger}(\beta_2)m_2 \otimes e^{\dagger}(\beta_1)m_1) \\
      &= (-1)^{\Par(m_1)\Par(m_2)}e^{\dagger}(\beta_1,\beta_2)\varphi_{w[n_2,n_1]}(e^{\dagger}(\beta_2)m_2 \otimes e^{\dagger}(\beta_1)m_1) \\
      &= (-1)^{\Par(m_1)\Par(m_2)}R_{F_{\beta_1}(M_1),F_{\beta_2}(M_2)}(e^{\dagger}(\beta_1)m_1 \otimes e^{\dagger}(\beta_2)m_2),
   \end{align*}
   which completes the proof.
\end{proof}

\begin{lem}\label{lem:compatibility of Morita and dual}
   For \(M \in RC_{\beta} \sgmod \), there is an isomorphism
   \[F_{\beta}(M^*) \simeq F_{\beta}(M)^*.\]
\end{lem}
\begin{proof}
   A direct computation yields
   \[e^{\dagger}(\beta).M^* = e^{\dagger}(\beta).\Hom_{\Bbbk}(M ,\Bbbk) = \Hom_{\Bbbk}(e^{\dagger}(\beta).M,\Bbbk) = F_{\beta}(M)^*.\qedhere\]
\end{proof}

\begin{dfn}
   For \(M \in RC_{\beta} \sgMod \), set \(M_{z} = \Bbbk[z] \otimes_{\Bbbk} M\) and equip it with the \(RC_{\beta}\)-action
   \begin{gather*}
      e(\nu).(f \otimes u) = f \otimes e(\nu).u, \\
      x_{k}e(\nu).(f \otimes u) = (-1)^{\pi_2(\nu_k)}zf \otimes e(\nu).u + f \otimes x_{k}e(\nu).u, \\
      c_{k}.(f(z) \otimes u) = f \otimes c_{k}.u,~\sigma_{a}.(f \otimes u) = f \otimes \sigma_{a}.u.
   \end{gather*}
   The module \(M_{z} \in RC_{\beta} \sgMod \) is called the \textit{affinization} of \(M\).
\end{dfn}
This provides a concrete example of the strong affinization introduced in Section \ref{subsec:affinization}.
Clearly,
   \[F_{\beta}(M_z) \simeq F_{\beta}(M)_z.\]
   For \(u \in M\), write \(u_z = 1 \otimes u \in M_z\).

\begin{dfn}\label{dfn:renormalized R-matrix}
   For \(M_i \in RC_{\beta_i} \sgmod \) (\(i=1,2\)), define \(r, s \in \mathbb{Z}_{\geq 0}\) by
   \begin{align*}
      r &= \max\{a \in \mathbb{Z}_{\geq 0} \mid \Im R_{(M_1)_z,M_2} \subset z^{a} M_2 \circ (M_1)_z \}, \\
      s &= \max \{b \in \mathbb{Z}_{\geq 0} \mid \Im R_{M_1, (M_2)_{z'}} \subset z'^{b} (M_2)_{z'} \circ M_1\}.
   \end{align*}

   The morphisms of \(RC_{\beta_1 + \beta_2}\)-modules \(r_{M_1,M_2}, r'_{M_1,M_2}\) are defined by
   \begin{align*}
      r_{M_1,M_2} &= z^{-r}R_{(M_1)_z,M_2}|_{z = 0} \colon M_1 \circ M_2 \to  M_2 \circ M_1,\\
      r'_{M_1,M_2} &= z'^{-s}R_{M_1, (M_2)_{z'}}|_{z' = 0} \colon M_1 \circ M_2 \to  M_2 \circ M_1,
   \end{align*}
   and are called the \textit{renormalized R-matrices}.
\end{dfn}
By Lemma \ref{lem:compatibility of Morita and R-matrix} and \cite[Proposition 1.10]{KKK}, non-zero \(M_1\) and \(M_2\) satisfy \(R_{(M_1)_z,M_2} \neq 0\) and \(R_{M_1,(M_2)_{z'}} \neq 0\), so that \(r\) and \(s\) are well defined.
If \(M \neq 0 \neq N\), then \(r_{M,N} \neq 0\).
Moreover, \(r_{M,N} = R_{M,N}\) whenever \(R_{M,N} \neq 0\).
\begin{prop}
   For \(m_i \in M_i \in RC_{\beta_i} \sgMod\) (\(i=1,2\)),
   \begin{align*}
      &F_{\beta_1 + \beta_2}(r_{M_1,M_2})(e^{\dagger}(\beta_1)m_1 \otimes e^{\dagger}(\beta_2)m_2) \\
      &= (-1)^{\Par(m_1)\Par(m_2)}r_{F_{\beta_1}(M_1),F_{\beta_2}(M_2)}(e^{\dagger}(\beta_1)m_1 \otimes e^{\dagger}(\beta_2)m_2), \\
      &F_{\beta_1 + \beta_2}(r'_{M_1,M_2})(e^{\dagger}(\beta_1)m_1 \otimes e^{\dagger}(\beta_2)m_2) \\
      &= (-1)^{\Par(m_1)\Par(m_2)}r'_{F_{\beta_1}(M_1),F_{\beta_2}(M_2)}(e^{\dagger}(\beta_1)m_1 \otimes e^{\dagger}(\beta_2)m_2).
   \end{align*}
\end{prop}
\begin{proof}
   Let \(r', s' \in \mathbb{Z}_{\geq 0}\) be defined by
   \begin{align*}
      r' &= \max\{a \in \mathbb{Z}_{\geq 0} \mid \Im R_{F_{\beta_1}(M_1)_z,F_{\beta_2}(M_2)} \subset z^{a} F_{\beta_2}(M_2) \circ (F_{\beta_1}(M_1))_z \}, \\
      s' &= \max \{b \in \mathbb{Z}_{\geq 0} \mid \Im R_{F_{\beta_1}(M_1), (F_{\beta_2}(M_2))_{z'}} \subset z'^{b} (F_{\beta_2}(M_2))_{z'} \circ F_{\beta_1}(M_1)\},
   \end{align*}
   so that
   \begin{align*}
      r_{F_{\beta_1}(M_1),F_{\beta_2}(M_2)} &= z^{-r'}R_{(F_{\beta_1}(M_1))_z,F_{\beta_2}(M_2)}|_{z = 0} ,\\
      r'_{F_{\beta_1}(M_1),F_{\beta_2}(M_2)} &= z'^{-s'}R_{F_{\beta_1}(M_1), (F_{\beta_2}(M_2))_{z'}}|_{z' = 0}.
   \end{align*}
   Since \(z^{-r}R_{(M_1)_z,M_2}\) and \(z'^{-s}R_{M_1, (M_2)_{z'}}\) are well defined as morphisms of \(RC_{\beta_1 + \beta_2}\)-modules, Lemma \ref{lem:compatibility of Morita and R-matrix} implies \(r \leq r'\) and \(s \leq s'\).
   A symmetric argument using a quasi-inverse of \(F\) yields \(r \geq r'\) and \(s \geq s'\); hence \(r = r'\) and \(s = s'\).
      Applying Lemma \ref{lem:compatibility of Morita and R-matrix} once more, the claim follows.
\end{proof}

\begin{rem}
   By \cite[Definition 1.11]{KKK}, \(r_{F_{\beta_1}(M_1),F_{\beta_2}(M_2)}\) and \(r'_{F_{\beta_1}(M_1),F_{\beta_2}(M_2)}\) agree up to a scalar multiple. Consequently, \(r_{M_1,M_2}\) and \(r'_{M_1,M_2}\) also agree up to a scalar multiple.
\end{rem}

\begin{rem}[{\cite[(1.20)]{KKK}}]
   As is the case for \(\{R_{M_1,M_2}\}\), the family \(\{r_{M_1,M_2}\}\) satisfies the Yang--Baxter equation.
   For any \(w \in \mathfrak{S}_{n}\) with reduced expression \(w = s_{i_1} \cdots s_{i_{l}}\), the morphism \(r^{w}_{M_1,\ldots,M_l}\) defined inductively below is independent of the choice of reduced expression:
   \begin{align*}
      r^{e}_{M_1,\ldots,M_l} &\coloneqq \id, \\
      r^{w}_{M_1,\ldots,M_l} &\coloneqq r^{ws_{i_l}}_{M_{s_{i_l}(1)},\ldots,M_{s_{i_{l}}(l)}} \cdot (M_1 \circ \cdots \circ M_{i_{l-1}} \circ r_{M_{i_{l}},M_{i_{l+1}}} \circ M_{i_{l+2}} \circ \cdots \circ M_{l}).
   \end{align*}
\end{rem}

By applying the results of Section \ref{sec:quiver hecke--clifford superalgebras} and following the arguments of \cite{KKKO1,KKKO2}, one can establish analogues for quiver Hecke--Clifford superalgebras of the results proved therein for symmetric quiver Hecke algebras.
We omit the proofs and collect some of the results below.

\begin{cor}[cf. {\cite[Corollary 3.9]{KKKO1}}]\label{cor:real simple module of quiver Hecke--Clifford superalgebra}
   Let \(RC_{\beta_i}\) (\(i= 1,2\)) be symmetric quiver Hecke--Clifford superalgebras, and let \(M_i \in RC_{\beta_i} \sgmod\) be simple modules, at least one of which is real simple.
   \begin{enumerate}[(i)]
      \item If \(\hd M \circ N \simeq \soc N \circ M\), then \(M \circ N\) is simple and \(M \circ N \simeq N \circ M\).
      \item \(M \circ N \simeq N \circ M\) if and only if \(M \circ N\) is simple.
   \end{enumerate}
\end{cor}

\begin{dfn}[cf. {\cite[Definition 3.1.1, Definition 3.2.2]{KKKO2}}]
   For non-zero \(M, N \in RC \sgmod\), define \(\Lambda(M,N) \in \mathbb{Z}\) to be the homogeneous degree of \(r_{M,N}\).
   Further define \(\mathfrak{d}(M,N) \in \mathbb{Z}\) by
   \[\mathfrak{d}(M,N) = \frac{1}{2}(\Lambda(M,N) + \Lambda(N,M)).\]
\end{dfn}

\begin{prop}[cf. {\cite[Lemma 3.2.3]{KKKO2}}]
   Let \(M, N \in RC \sgmod\) be simple modules, at least one of which is real.
   The following are equivalent:
   \begin{enumerate}[(i)]
      \item \(\mathfrak{d}(M,N) = 0\),
      \item \(r_{M,N}\) and \(r_{N,M}\) are inverse to each other up to a scalar,
      \item \(M \circ N \simeq N \circ M\) in \(RC \smod\),
      \item \(\hd(M \circ N) \simeq \hd(N \circ M)\) in \(RC \smod\),
      \item \(M \circ N\) is simple.
   \end{enumerate}
\end{prop}

\section{Schur--Weyl duality via quiver Hecke--Clifford superalgebra}\label{sec:schur--weyl duality via quiver hecke--clifford superalgebra}

We apply the results of Section \ref{sec:quiver hecke--clifford superalgebras} to construct a Schur--Weyl duality via quiver Hecke--Clifford superalgebras and to study its properties.

The following notation and conventions will be used throughout this section.
They follow the notation of \cite{KL4}.

\begin{enumerate}[$\bullet$]
   \item \(\Bbbk = \overline{\mathbb{Q}(q)}\),
   \item \(\epsilon = (\epsilon_1,\ldots,\epsilon_n) \in (\mathbb{Z}/2\mathbb{Z})^n~(n \geq 4)\),
   \item \(M = |\{i \mid \varepsilon_i = 0 \}|,~N = |\{i \mid \varepsilon_i = 1\}|\),
   \item \(\epsilon_{M|N} = (\underbrace{\overline{0},\ldots,\overline{0}}_{M},\underbrace{\overline{1},\ldots,\overline{1}}_{N})\),
   \item \(\mathbb{I} = \{1 < 2 < \cdots < n\}\),
   \item \(\mathbb{I}_{\overline{j}} = \{i \in \mathbb{I} \mid \epsilon_i = j\}~(j = 0,1)\),
   \item \(P = \bigoplus_{i \in \mathbb{I}}\mathbb{Z}\delta_i\),
   \item \(P_{\geq 0} = \bigoplus_{i \in \mathbb{I}}\mathbb{Z}_{\geq 0}\delta_i \subset P\),
   \item \((-|-) \colon\) the bilinear form on \(P\) defined by \((\delta_i|\delta_j) = \delta_{ij}\),
   \item \(P^{\vee} = \Hom_{\mathbb{Z}}(P,\mathbb{Z}) = \bigoplus_{i \in \mathbb{I}}\mathbb{Z}\delta_{i}^{\vee},~\langle \delta_i,\delta_{j}^{\vee}\rangle = \delta_{ij}\),
   \item \(I = \{0,1,\ldots,n-1\}\), where \(i \in I\) is taken modulo \(n\),
   \item \(\alpha_i = \epsilon_i - \epsilon_{i+1} \in P~(i \in I)\), with \(\alpha_{n-1} = \epsilon_{n-1} - \epsilon_{0}\),
   \item \(\alpha_{i}^{\vee} = \delta_{i}^{\vee} - (-1)^{\epsilon_i + \epsilon_{i+1}}\delta_{i+1}^{\vee} \in P^{\vee}~(i \in I)\),
   \item \(I_{\mathsf{even}} = \{i \in I \mid (\alpha_i|\alpha_i) = \pm 2\},~I_{\mathsf{odd}} = \{i \in I \mid (\alpha_i|\alpha_i) = 0\}\),
   \item \(q_i = (-1)^{\epsilon_i}q^{(-1)^{\epsilon_i}}~(i \in I)\),
   \item \(\mathbf{q}(-,-) \colon P \times P \to \Bbbk^{\times},~\mathbf{q}(\mu,\nu) = \Pi_{i \in \mathbb{I}}q_{i}^{\langle \mu,\delta_i^{\vee} \rangle \langle \nu,\delta_i^{\vee} \rangle}\).
\end{enumerate}

\subsection{Odd involutions on the affinization of vector representations}\label{subsec:Odd involutions on the affinization of vector representations}

In this section we review two types of affine queer Lie superalgebras and the vector representations of their quantum analogues, following \cite{GS,CG}.

\subsubsection{\(\qq_n^{(1)},\qq_n^{(2)}\)}

Let \(n \geq 1\).
For a super vector space \(V = \Bbbk^{n|n} \in \Bbbk \sVect\), \(\End(V)\) forms a Lie superalgebra under the bracket
\[[X,Y] = XY - (-1)^{\Par(X)\Par(Y)}YX,\]
denoted \(\gll_{n|n}\).
The \textit{queer Lie superalgebra} \(\qq_n\) is the Lie subsuperalgebra of \(\gll_{n|n}\) defined by
\[\qq_n = \{X \in \gll_{n|n} \mid [X,J] = 0\},\]
where
\[J = \sqrt{-1}\left[ \begin{array}{c|c}
                                          0 & I_n \\
                                          \hline
                                    - I_n & 0
                                       \end{array}
                                    \right].\]
Set \(\sqq_n = [\qq_n,\qq_n]\).
The loop algebra \(\mathcal{L}(\sqq_n)\) of \(\sqq_n\) is given by
\[\mathcal{L}(\sqq_n) = \sqq_n \otimes \Bbbk[t^{\pm 1}].\]
The \textit{untwisted affine queer Lie superalgebra} \(\qq_n^{(1)}\) is the extension of \(\mathcal{L}(\sqq_n)\) by a derivation:
\[\qq_n^{(1)} = \mathcal{L}(\sqq_n)\oplus \Bbbk.D,\]
where
\[[D,X \otimes t^k] = k X \otimes t^k.\]

Define an automorphism \(\varepsilon\) on \(\sqq_n\) by \(\varepsilon|_{({\sqq_n})_{i}} = (-1)^i \id_{({\sqq_n})_{i}}\), and extend it to \(\qq_n^{(1)}\) by setting
\[\varepsilon(t) = -1, \quad \varepsilon(D) = D.\]
The \textit{twisted affine queer Lie superalgebra} \(\qq_n^{(2)}\) is the quotient of the sub-Lie superalgebra
\[(\qq_n^{(1)})^{\varepsilon} = (\mathcal{L}(\sqq_n))^{\varepsilon} \oplus \Bbbk.D\]
by the ideal \(\sum_{i \neq 0}\Bbbk \otimes t^{2i}\).
Note that \(\qq_n\) and \(\qq_n^{(1)}\) are not Kac--Moody superalgebras, whereas \(\qq_n^{(2)}\) is a non-symmetrizable Kac--Moody superalgebra.

For \(V\), set \(V_z \coloneqq V \otimes \Bbbk[z^{\pm 1}]\), and let \(\qq_n^{(1)}\) and \((\qq_n^{(1)})^{\varepsilon}\) act on \(V_z\) by
\[X \otimes t^k.v \otimes t^l = (X v) \otimes t^{k+l}, \quad D.v \otimes t^l = l v \otimes t^l.\]
Then \(J \otimes 1 \in \End(V_z)\) is an odd involution and supercommutes with the actions of \(\qq_n^{(1)}\) and \((\qq_n^{(1)})^{\varepsilon}\).

\begin{rem}
   Since the ideal \(\sum_{i \neq 0}\Bbbk \otimes t^{2i}\) does not act trivially on \(V_z\), \(V_z\) cannot be regarded as a \(\qq_n^{(2)}\)-module.
\end{rem}

\subsubsection{\(U_q(\widehat{\qq}_n^{\mathsf{tw}})\)}

Fix the standard basis \(e_{-n},\ldots,e_{-1},e_{1},\ldots,e_{n}\) of \(V = \Bbbk^{n|n}\).
The parity assignment is
\[\Par(e_i) = \begin{cases}
   \overline{0} & \text{if $i > 0$,}  \\
   \overline{1} & \text{if $i < 0$.}
\end{cases}\]
With respect to this basis, \(J\) takes the form
\[J = \sum_{i = 1}^n (E_{i,-i} - E_{-i,i}),\]
where \(E_{ij}\) denote the matrix units.

The \textit{Olshanski R-matrix} \(S \in \End(V)^{\otimes 2}\) is defined by
\begin{align*}
   S &= \sum_{i,j = -n}^n q^{(\delta_{i,j} + \delta_{-i,j})(1-2\Par(e_j))}E_{ii} \otimes E_{jj} \\
   &+ (q-q^{-1})\sum_{i > j}(-1)^{\Par(j)} (E_{ij} + E_{-i,-j})\otimes E_{ji}.
\end{align*}
Set
\[S(x,y) = S + \frac{(q-q^{-1})P}{x^{-1}y-1} + \frac{(q-q^{-1})PJ_{1}J_{2}}{xy-1} \in \End(V)^{\otimes 2} \otimes \Bbbk(x,y),\]
where \(P = \sum_{i,j = -n}^{n}(-1)^{\Par(e_j)}E_{ij} \otimes E_{ji}\) denotes the superpermutation operator.
\(S(x,y)\) satisfies the Yang--Baxter equation
\[S_{12}(x,y) S_{13}(x,z) S_{23}(y,z) = S_{23}(y,z) S_{13}(x,z) S_{12}(x,y).\]

We define \(U_q(\widehat{\qq}_n^{\mathsf{tw}})\) as the superalgebra generated by \(\{t_{ij}^{(r)} \mid \cdots\}\) and an even central element \(C\), subject to the following relations:
\[t_{ij}^{(0)} = 0 ~(i > j), \quad t_{ii}^{(0)}t_{-i,-i}^{(0)} = 1 = t_{-i,-i}^{(0)}t_{i,i}^{(0)},\]
\[T_{12}(w)T_{13}(z)S_{23}(C^{-1}w,C^{-1}z) = S_{23}(Cw,Cz)T_{13}(z)T_{12}(w),\]
where
\[T(z) = \sum_{i,j} t_{ij}(z) \otimes E_{ij} \in U_q(\widehat{\qq}_n^{\mathsf{tw}})[[z^{-1}]] \otimes \End(V), \quad t_{ij}(z) = \sum_{r \geq 0} t_{ij}^{(r)} z^{-r}.\]
The quotient \(U_q(\mathcal{L}_{\mathsf{tw}}\qq_n) = U_q(\widehat{\qq}_n^{\mathsf{tw}})/\Bbbk.C\) is called the \textit{twisted quantum loop superalgebra}.

\begin{prop}[Proof of {\cite[Theorem 5.1]{CG}}]
   \(V_z\) admits a \(U_q(\mathcal{L}_{\mathsf{tw}}\qq_n)\)-module structure via the action
   \[\rho(T(w)) = S(z,w) \in \End(V_z) \otimes \End(V)[[w^{-1}]] \otimes \End(V).\]
   Explicitly,
   \[\rho(t_{ij}^{(r)}) = \begin{cases}
      S_{ij} & (r = 0), \\
      (-1)^{\Par(e_i)} \epsilon (E_{ji}w^{-r} + E_{-j,-i}w^{r}) & (r > 0). \\
      \end{cases}\]
\end{prop}
Setting \(s \colon \Bbbk[z^{\pm 1}] \to \Bbbk[z^{\pm 1}]; z \mapsto z^{-1}\), the map \(J \otimes s \in \End(V_z)\) is a morphism of \(U_q(\mathcal{L}_{\mathsf{tw}}\qq_n)\)-modules.

\subsection{Schur--Weyl duality for generalized quantum group of type \(A\)}

\begin{dfn}[{\cite[Definition 2.1]{KL4}}]
   The superalgebra \(\mathcal{U}(\epsilon)\) over \(\Bbbk\) is the superalgebra generated by \(k_{\mu}, e_i, f_i~(\mu \in P, i \in I)\) with parities
   \[\Par(k_{\mu}) = 0,~\Par(e_i) = \Par(f_i) = \epsilon_i + \epsilon_{i+1},\]
   subject to the following relations:
   \begin{enumerate}[(i)]
      \item \(k_0 = 1,~k_{\mu + \mu'} = k_{\mu}k_{\mu'}~(\mu,\mu' \in P)\),
      \item \(k_{\mu} e_i k_{\mu}^{-1} = \mathbf{q}(\mu,\alpha_i)e_i,~k_{\mu} f_i k_{\mu}^{-1} = \mathbf{q}(\mu,\alpha_i)^{-1}f_i~(\mu \in P, i \in I)\),
      \item \(e_i f_j - f_j e_i = \delta_{ij}\dfrac{k_{\alpha_i} - k_{\alpha_i}^{-1}}{q_i - q_i^{-1}}~(i,j \in I)\),
      \item \(e_i^2 = f_i^2 = 0~(i \in I_{\mathsf{odd}})\),
      \item \(e_{i}e_{j} - e_{j}e_{i} = f_{i}f_{j} - f_{j}f_{i} = 0~(|i-j| \neq 1)\),
      \item \(
         \begin{aligned}
            e_{i}^2 e_{j} - (-1)^{\epsilon_i}[2]e_{i} e_{j} e_{i} + e_{j} e_{i}^2 = 0 \\
            f_{i}^2 f_{j} - (-1)^{\epsilon_i}[2]f_{i} f_{j} f_{i} + f_{j} f_{i}^2 = 0
         \end{aligned}
         ~(|i-j| = 1, i \in I_{\mathsf{even}})\),
      \item \(
         \begin{aligned}
            e_{i} e_{i-1} e_{i} e_{i+1} &- e_{i} e_{i+1} e_{i} e_{i-1} + e_{i+1} e_{i} e_{i-1} e_{i} \\
            &- e_{i-1} e_{i} e_{i+1} e_{i} + (-1)^{\epsilon_i}[2]e_{i} e_{i-1} e_{i+1} e_{i} = 0 \\
            f_{i} f_{i-1} f_{i} f_{i+1} &- f_{i} f_{i+1} f_{i} f_{i-1} + f_{i+1} f_{i} f_{i-1} f_{i} \\
            &- f_{i-1} f_{i} f_{i+1} f_{i} + (-1)^{\epsilon_i}[2]f_{i} f_{i-1} f_{i+1} f_{i} = 0
         \end{aligned}
         (i \in I_{\mathsf{odd}})\).
   \end{enumerate}
   \(\mathcal{U}(\epsilon)\) is called the \textit{generalized quantum group of affine type \(A\) associated with \(\epsilon\)}.
   From now on, we write \(k_i \coloneqq k_{\alpha_i}\).
\end{dfn}

In what follows, we consider the case \(\epsilon = \epsilon_{M|M}\) and \(n = 2M \geq 4\).

\begin{rem}
   The generalized Schur--Weyl duality between \(\mathcal{U}(\epsilon)\) and quiver Hecke algebras described below was established in \cite{KL4}, where the condition \(M \neq N\) was imposed for the convenience of constructing the normalized R-matrix from a universal R-matrix.
   In \cite{KOS,KY}, the normalized R-matrix on the fundamental module is constructed explicitly under the weaker assumption \(M + N \geq 4\) alone. By appealing to these works, the Schur--Weyl duality may be established even when \(M = N\).
\end{rem}

\begin{rem}
   \(\mathcal{U}(\epsilon)\) is a Hopf superalgebra, whose comultiplication \(\Delta\) is given by
   \begin{align*}
      &\Delta(k_{\mu}) = k_{\mu} \otimes k_{\mu}, \\
      &\Delta(e_i) = 1 \otimes e_i + e_i \otimes k_{\alpha_i}^{-1}, \\
      &\Delta(f_i) = f_i \otimes 1 + k_i \otimes f_i.
   \end{align*}
\end{rem}

Let \(\mathring{\mathcal{U}}(\epsilon)\) denote the subalgebra of \(\mathcal{U}(\epsilon)\) generated by
   \[\mathring{\mathcal{U}}(\epsilon) = \langle k_{\mu},e_i,f_i \mid \mu \in P, i \in I\backslash\{0\}\rangle_{\Bbbk \text{-subalg}}.\]

\begin{dfn}
   For \(V \in \mathcal{U}(\epsilon) \sMod\) and \(\lambda \in P\), the \textit{weight space} \(V_{\lambda}\) of weight \(\lambda\) is defined by
   \[V_{\lambda} = \{v \in V \mid k_{\mu}v = \mathbf{q}(\lambda,\mu)v\}.\]
   The full subcategory \(\mathcal{C}(\epsilon)\) of \(\mathcal{U}(\epsilon) \smod \) consists of those modules \(V\) satisfying
   \[V = \bigoplus_{\lambda \in P_{\geq 0}}V_{\lambda};\]
   the objects of \(\mathcal{C}(\epsilon)\) are called \textit{polynomial modules}.
   The full subcategory \(\mathring{\mathcal{C}}(\epsilon)\) of \(\mathring{\mathcal{U}}(\epsilon) \smod \) is defined analogously.
\end{dfn}

   The set \(\mathbb{Z}_{\geq 0}^{n}(\epsilon)\) is defined by
   \[\mathbb{Z}_{\geq 0}^{n}(\epsilon) = \{\mathbf{m} = (m_1,\ldots,m_n) \in \mathbb{Z}_{\geq 0}^{n} \mid \epsilon_i = 1 \Rightarrow m_i \in \{0,1\}\}.\]
   For \(\mathbf{m} \in \mathbb{Z}_{\geq 0}^{n}(\epsilon)\), set
   \[|\mathbf{m}| = m_1 + \cdots + m_n.\]
   For \(i \in \mathbb{I}\), let
   \[\mathbf{e}_i = (0,\ldots,1,\ldots,0) \in \mathbb{Z}_{\geq 0}^{n}(\epsilon)\]
   denote the vector with \(1\) in the \(i\)-th component and \(0\) elsewhere.
   For \(l \in \mathbb{Z}_{\geq 0}\), the super vector space \(\mathcal{W}_{l,\epsilon}\) is defined by
   \begin{align*}
      \mathcal{W}_{l,\epsilon} &= \bigoplus_{\mathbf{m} \in \mathbb{Z}_{\geq 0}^{n}(\epsilon), |\mathbf{m}| = l}\Bbbk | \mathbf{m} \rangle, \\
      \Par(|\mathbf{m}\rangle) &\equiv m_{M+1} + \cdots + m_n~\bmod 2.
   \end{align*}
   For each scalar \(x \in \Bbbk^{\times}\), \(\mathcal{U}(\epsilon)\) acts on \(\mathcal{W}_{l,\epsilon}\) via
   \begin{align*}
      k_{\mu}|\mathbf{m} \rangle &= \mathbf{q}(\mu,\sum_{i \in \mathbb{I}}m_i \delta_i)|\mathbf{m}\rangle, \\
      e_i |\mathbf{m} \rangle &= \begin{cases}
         x^{\delta_{i0}}[m_{i+1}]|\mathbf{m} + \mathbf{e}_i - \mathbf{e}_{i+1}\rangle & \text{if $\mathbf{m} + \mathbf{e}_i - \mathbf{e}_{i+1} \in \mathbb{Z}_{\geq 0}^{n}(\epsilon)$,} \\
         0 & \text{otherwise;}
      \end{cases} \\
      f_i |\mathbf{m} \rangle &= \begin{cases}
         x^{-\delta_{i0}}[m_{i}]|\mathbf{m} - \mathbf{e}_i + \mathbf{e}_{i+1}\rangle & \text{if $\mathbf{m} - \mathbf{e}_i + \mathbf{e}_{i+1} \in \mathbb{Z}_{\geq 0}^{n}(\epsilon)$,} \\
         0 & \text{otherwise.}
      \end{cases}
   \end{align*}
   This module is denoted \(\mathcal{W}_{l,\epsilon}(x)\); for every \(x \in \Bbbk^{\times}\), we have \(\mathcal{W}_{l,\epsilon}(x) \in \mathcal{C}(\epsilon)\).
   The module \(\mathcal{W}_{l,\epsilon}(x)\) is called the \textit{fundamental module with spectral parameter \(x\)}.
   We write \(\mathcal{W}_{l,\epsilon} = \mathcal{W}_{l,\epsilon}(1)\) for short.
\begin{lem}[{\cite[Section 5.1]{KY}}]
   For \(l,m \in \mathbb{Z}_{\geq 0}\) and \(x,y \in \Bbbk^{\times}\), there is a decomposition of \(\mathring{\mathcal{U}}(\epsilon)\)-modules
   \[\mathcal{W}_{l,\epsilon}(x) \otimes_{\Bbbk} \mathcal{W}_{m,\epsilon}(y) \simeq \bigoplus_{t \in H(l,m)}V_{\epsilon}((l + m - t,t)),\]
   where
   \[H(l,m) = \{t \in \mathbb{Z} \mid 0 \leq t \leq \min\{l,m\},~(l+m-t,t) \in \mathcal{P}_{M|N}\}.\]
\end{lem}

There is a unique highest weight vector \(v(l,m,t) \in V_{\epsilon}((l + m - t,t)) \subset \mathcal{W}_{l,\epsilon}(x) \otimes_{\Bbbk} \mathcal{W}_{m,\epsilon}(y)\) satisfying the following conditions:
   \begin{enumerate}[(i)]
      \item \(v(l,m,t) \in \mathcal{L}_{l,\epsilon} \otimes \mathcal{L}_{m,\epsilon}\),
      \item \(v(l,m,t) \equiv |(l-t)\mathbf{e}_1 + t\mathbf{e}_2 \rangle \otimes |m\mathbf{e}_1 \rangle \bmod q\mathcal{L}_{l,\epsilon} \otimes \mathcal{L}_{m,\epsilon}\),
   \end{enumerate}
   where
   \begin{align*}
      \mathcal{L}_{l,\epsilon} &= \bigoplus_{\mathbf{m} \in \mathbb{Z}_{\geq 0}^{n}(\epsilon), |\mathbf{m}| = l}\mathbb{A}_0 |\mathbf{m}\rangle, \\
      \mathbb{A}_0 &= \{f(q)/g(q) \mid f(q),g(q) \in \mathbb{Q}[q], g(0) \neq 0\}.
   \end{align*}
   A morphism of \(\mathring{\mathcal{U}}(\epsilon)\)-modules
   \[\mathcal{P}_{t}^{l,m} \colon \mathcal{W}_{l,\epsilon} \otimes \mathcal{W}_{m,\epsilon} \to \mathcal{W}_{m,\epsilon} \otimes \mathcal{W}_{l,\epsilon}\]
   is then defined by
   \[\mathcal{P}_{t}^{l,m}(v(l,m,t')) = \delta_{tt'}v(l,m,t).\]
\begin{dfn}
   Let \(V \in \mathcal{C}(\epsilon)\) and let \(z\) be an indeterminate. The space
   \[V_{\text{aff}} = \Bbbk[z^{\pm1}] \otimes V\]
   carries a \(\mathcal{U}(\epsilon)\)-module structure given by
   \begin{align*}
      k_{\mu}.(1 \otimes v) &= 1 \otimes (k_{\mu}v), \\
      e_i.(1 \otimes v) &= z^{\delta_{i,0}} \otimes (e_i v), \\
      f_i.(1 \otimes v) &= z^{-\delta_{i,0}} \otimes (f_i v).
   \end{align*}
   This module is called the \textit{affinization} of \(V\).
\end{dfn}
\begin{dfn}[{\cite[Theorem 3.15]{KL3},\cite[Theorem 5.2]{KY},\cite{KOS}}]
   The map
   \begin{align*}
      R_{l,m}(z_{1}/z_{2}) &\colon \Bbbk(z_1,z_2) \otimes_{\Bbbk[z_1^{\pm1},z_2^{\pm1}]} (\mathcal{W}_{l,\epsilon})_{\text{aff}} \otimes (\mathcal{W}_{m,\epsilon})_{\text{aff}} \\
      &\to \Bbbk(z_1,z_2) \otimes_{\Bbbk[z_1^{\pm1},z_2^{\pm1}]} (\mathcal{W}_{m,\epsilon})_{\text{aff}} \otimes (\mathcal{W}_{l,\epsilon})_{\text{aff}}
   \end{align*}
   defined by
   \begin{align*}
      R_{l,m}(z_1/z_2) &= \sum_{t \in H(l,m)} \rho_{t}(z_1/z_2) \mathcal{P}_{t}^{l,m}, \\
      \rho_t(z) &= \prod_{i=t+1}^{\min\{l,m\}}\dfrac{1 - q^{l+m-2i+2}z}{z - q^{l+m-2i+2}}
   \end{align*}
   is a morphism of \(\Bbbk(z_1,z_2) \otimes \mathcal{U}(\epsilon)\)-modules.
   Here \(\rho_{\min\{l,m\}}(z) = 1\).
   The map is called the \textit{normalized R-matrix} and denoted by \(R_{l,m}^{\mathrm{norm}}\).
\end{dfn}
\begin{rem}[{\cite[Corollary 3.14]{KL3},\cite[Corollary 4.8,\(\S 3.1\)]{KY}}]\label{rem:computation of normalized R-matrix}
   Consider the case \(l = m = 1\).
   Then \(\mathcal{W}_{1,\epsilon} \simeq \Bbbk^{M|M}\), and the \(\mathring{\mathcal{U}}(\epsilon)\)-module decomposition reads
   \begin{align*}
      \mathcal{W}_{1,\epsilon} \otimes \mathcal{W}_{1,\epsilon} \simeq V_{\epsilon}((2)) \oplus V_{\epsilon}((1,1)).
   \end{align*}
   The normalized R-matrix is
   \[R_{1,1}(z) = \dfrac{1 - q^{2}z}{z - q^{2}} \mathcal{P}_{0}^{1,1} + \mathcal{P}_{1}^{1,1},\]
   where
   \begin{align*}
      \mathcal{P}_{0}^{1,1} &\colon \mathcal{W}_{1,\epsilon} \otimes \mathcal{W}_{1,\epsilon} \to V_{\epsilon}((2)), \\
      \mathcal{P}_{1}^{1,1} &\colon \mathcal{W}_{1,\epsilon} \otimes \mathcal{W}_{1,\epsilon} \to V_{\epsilon}((1,1))
   \end{align*}
   are the canonical projections.

   Define a linear map \(\mathcal{R}\) on \(\mathcal{W}_{1,\epsilon} \otimes \mathcal{W}_{1,\epsilon}\) by
   \[\mathcal{R}(| \mathbf{e}_i \rangle \otimes | \mathbf{e}_j \rangle) = \begin{cases}
      q^{-1} q_{i}^{-1} | \mathbf{e}_i \rangle \otimes | \mathbf{e}_i \rangle & \text{if $i = j$,} \\
      q^{-1} | \mathbf{e}_j \rangle \otimes | \mathbf{e}_i \rangle & \text{if $i > j$,} \\
      (q^{-2} - 1)| \mathbf{e}_i \rangle \otimes | \mathbf{e}_j \rangle + q^{-1} | \mathbf{e}_j \rangle \otimes | \mathbf{e}_i \rangle & \text{if $i < j$.}
   \end{cases}\]
   The map \(\mathcal{R}\) is an isomorphism of \(\mathring{\mathcal{U}}(\epsilon)\)-modules satisfying
   \[(\mathcal{R} - q^{-2})(\mathcal{R} + 1) = 0.\]
   Its eigenvalues are \(q^{-2}\) and \(-1\), and the corresponding eigenspace decomposition recovers the direct sum decomposition of \(\mathcal{W}_{1,\epsilon} \otimes \mathcal{W}_{1,\epsilon}\) as a \(\mathring{\mathcal{U}}(\epsilon)\)-module.
   Explicitly, the irreducible decomposition
   \[\mathcal{W}_{1,\epsilon} \otimes \mathcal{W}_{1,\epsilon} = V_{\epsilon}((2)) \oplus V_{\epsilon}((1,1))\]
   identifies each summand with an eigenspace of \(\mathcal{R}\).
   Since \(| \mathbf{e}_1 \rangle \otimes | \mathbf{e}_1 \rangle \in V((2))\), the irreducible summands correspond to the eigenspaces of \(\mathcal{R}\) as follows:
   \begin{align*}
      V((2)) &= \{v \otimes w \mid \mathcal{R}(u \otimes w) = q^{-2} u \otimes v \} \\
      &= \bigoplus_{i = 1}^{M} \Bbbk.| \mathbf{e}_i \rangle \otimes | \mathbf{e}_i \rangle \oplus \bigoplus_{i < j} \Bbbk.(q^{-1}| \mathbf{e}_i \rangle \otimes | \mathbf{e}_j \rangle + | \mathbf{e}_j \rangle \otimes | \mathbf{e}_i \rangle), \\
      V((1,1)) &= \{v \otimes w \mid \mathcal{R}(u \otimes w) = - u \otimes v \} \\
      &= \bigoplus_{i = M+1}^{M+N} \Bbbk.| \mathbf{e}_i \rangle \otimes | \mathbf{e}_i \rangle \oplus \bigoplus_{i < j} \Bbbk.(| \mathbf{e}_i \rangle \otimes | \mathbf{e}_j \rangle -q^{-1} | \mathbf{e}_j \rangle \otimes | \mathbf{e}_i \rangle).
   \end{align*}
   The normalized R-matrix \(R_{1,1}^{\mathrm{norm}}(z)\) is then given by
   \[R_{1,1}(z)^{\mathrm{norm}}(| \mathbf{e}_i \rangle \otimes | \mathbf{e}_j \rangle) = \begin{cases}
      \frac{1-q^{2}z}{z-q^{2}} | \mathbf{e}_i \rangle \otimes | \mathbf{e}_i \rangle & \text{if $i = j \leq M$,} \\
      | \mathbf{e}_i \rangle \otimes | \mathbf{e}_i \rangle & \text{if $i = j > M$,} \\
      \frac{1-q^2}{z-q^2}| \mathbf{e}_i \rangle \otimes | \mathbf{e}_j \rangle + \frac{q(1-z)}{z-q^2}| \mathbf{e}_j \rangle \otimes | \mathbf{e}_i \rangle & \text{if $i < j$,} \\
      \frac{z(1-q^2)}{z-q^2}| \mathbf{e}_i \rangle \otimes | \mathbf{e}_j \rangle + \frac{q(1-z)}{z-q^2}| \mathbf{e}_j \rangle \otimes | \mathbf{e}_i \rangle & \text{if $i > j$.}
   \end{cases}\]
\end{rem}

Set \(I = \mathbb{Z}\), and define a map \(X \colon I \to \Bbbk^{\times}\) by
\[X(i) = q^{-2i}.\]
The ratio \(X(i)/X(i+1)\) coincides with the pole of \(R_{1,1}^{\mathrm{norm}}\).
For \(i,j \in I\), let \(d_{ij}\) denote the multiplicity of \(X(j)/X(i)\) as a root of \(z - q^2\). Setting \(a_{ij} = - d_{ij} - d_{ji}\) recovers the following Cartan matrix of type \(A_{\infty}\):
\[[a_{ij}]_{i,j \in I} = \begin{bmatrix}
         \ddots & \vdots & \vdots & \vdots & \vdots & \rotatebox{70}{$\ddots$} \\
         \cdots & 2 & -1 & 0 & 0 & \cdots \\
         \cdots & -1 & 2 & -1 & 0 & \cdots \\
         \cdots & 0 & -1 & 2 & -1 & \cdots \\
         \cdots & 0 & 0 & -1 & 2 & \cdots \\
         \rotatebox{70}{$\ddots$} & \vdots & \vdots & \vdots & \vdots & \ddots
      \end{bmatrix}\]
The free commutative monoid \(\mathsf{Q}^{+}\) is defined by
   \[\mathsf{Q}^{+} \coloneqq \bigoplus_{i \in I}\mathbb{Z}_{\geq 0}.\alpha_{i},\]
   and a family of polynomials \(\{Q_{ij}(u,v)\}_{i,j \in I}\) by
   \[Q_{ij}(u,v) = \begin{cases}
      0 & \text{if $i = j$,} \\
      (u-v)^{-a_{ij}} & \text{if $i > j$,} \\
      (v-u)^{-a_{ij}} & \text{if $i < j$.} \\
   \end{cases}\]
From these data, we construct the quiver Hecke algebra \(R_{\beta}\) of type \(A_{\infty}\).
Let \(\mathcal{P}_{\beta}\) denote its polynomial representation (cf. \cite[(5.7)]{KL4}).
For \(\nu \in I^{\beta}\), let \(\mathcal{P}_{\nu}\) be the subalgebra of \(\mathcal{P}_{\beta}\) generated by \(e(\nu)\) and \(\{x_i\}_{1 \leq i \leq n}\),
and write \(\mathbb{O}_{\nu}\) for the completion of \(\mathcal{P}_{\nu}\) at \(X\):
\[\begin{array}{ccc}
    \mathcal{P}_{\nu} & \hookrightarrow & \mathbb{O}_{\nu} = \Bbbk[[X_1 - X(\nu_1),\ldots,X_n - X(\nu_n)]]e(\nu) \\
    \rotatebox{90}{$\in$} & & \rotatebox{90}{$\in$} \\
    x_i e(\nu) & \longmapsto & \frac{1}{x(\nu_i)}(X_i - X(\nu_i))e(\nu) \\
   \end{array}\]
Let \(\mathbb{K}_{\nu}\) be the quotient field of \(\mathbb{O}_{\nu}\), and write \(\mathbb{O}_{\beta}\) and \(\mathbb{K}_{\beta}\) for the direct sums over all \(\nu \in I^{\beta}\).
Since \(\mathfrak{S}_n\) acts on \(\mathbb{K}_{\beta}\), the semidirect product algebra \(\mathbb{K}_{\beta} \rtimes \Bbbk[\mathfrak{S}_n]\) is well defined and yields an embedding \(R_{\beta} \hookrightarrow \mathbb{K}_{\beta} \rtimes \Bbbk[\mathfrak{S}_n]\).

For each \(\nu \in I^{\beta}\) and \(1 \leq i \leq n\), set
\[V_{\nu_i} = (\mathcal{W}_{1,\epsilon})_{\text{aff}},\]
and
\begin{align*}
   V_{\nu} &= V_{\nu_1} \otimes \cdots \otimes V_{\nu_n}.
   \end{align*}
Let \(K = \Bbbk[X_1^{\pm1},\ldots,X_n^{\pm1}]\). Identifying \(X_i\) with the indeterminate \(z_i\) of the affinization of \(V_{\nu_i}\), we obtain
\[V_{\nu} \simeq K \otimes_{\Bbbk} (\mathcal{W}_{1,\epsilon})^{\otimes n}.\]
Define
\begin{align*}
   V^{\otimes \beta} &= \bigoplus_{\nu \in J^{\beta}}V_{\nu}e(\nu), \\
   V_{\mathbb{O}}^{\otimes \beta} &= \mathbb{O}_{\beta} \otimes_{K} V^{\otimes \beta}, \\
   V_{\mathbb{K}}^{\otimes \beta} &= \mathbb{K}_{\beta} \otimes_{\mathbb{O}_{\beta}} V_{\mathbb{O}}^{\otimes \beta} = \mathbb{K}_{\beta} \otimes_{K} V^{\otimes \beta}.
\end{align*}
Then \(V_{\mathbb{K}}^{\otimes \beta}\) carries a right action of \(\mathbb{K}_{\beta} \rtimes \Bbbk[\mathfrak{S}_n]\), where \(1 \otimes s_i\) acts by
\[\begin{array}{ccc}
    \mathbb{K}_{\nu}e(\nu) \otimes_{\Bbbk} V_{\nu} & \longrightarrow & \mathbb{K}_{s_k.\nu}e(s_k.\nu) \otimes_{\Bbbk} V_{s_k.\nu} \\
    \rotatebox{90}{$\in$} & & \rotatebox{90}{$\in$} \\
    fe(\nu) \otimes (v_1 \otimes \cdots \otimes v_n) & \longmapsto & (s_{k}.f)e(s_i.\nu) \otimes (\cdots \otimes R_{1,1}^{\mathrm{norm}}(v_i \otimes v_{i+1}) \otimes \cdots) \\
\end{array}\]
This makes \(V_{\mathbb{K}}^{\otimes \beta}\) a right \(R_{\beta}\)-module.
Since \(R_{1,1}^{\mathrm{norm}}\) is a morphism of \(\mathcal{U}(\epsilon)\)-modules, \(V_{\mathbb{K}}^{\otimes \beta}\) is in fact a \((\mathcal{U}(\epsilon),R_{\beta})\)-bimodule.

\begin{thm}[{\cite[Proposition 5.4]{KL4}}]\label{thm:Schur--Weyl subbimod}
 \(V_{\mathbb{O}}^{\otimes \beta}\) is a \((\mathcal{U}(\epsilon),R_{\beta})\)-subbimodule of \(V_{\mathbb{K}}^{\otimes \beta}\).
\end{thm}

Let \(\mathcal{F}_{\epsilon,\beta}\) denote the functor
\[V_{\mathbb{O}}^{\otimes \beta} \otimes_{R_{\beta}} - \colon R_{\beta} \gMod  \to \mathcal{U}(\epsilon) \sMod ,\]
and set
   \begin{align*}
      \mathcal{F}_{\epsilon,n} &= \bigoplus_{|\beta| = n} \mathcal{F}_{\epsilon,\beta} \colon \bigoplus_{|\beta| = n} R_{\beta} \gMod  \to \mathcal{U}(\epsilon) \sMod,  \\
      \mathcal{F}_{\epsilon} &= \bigoplus_{n \geq 0}\mathcal{F}_{\epsilon,n} \colon \bigoplus_{\beta \in \mathsf{Q}^{+}} R_{\beta} \gMod  \to \mathcal{U}(\epsilon) \sMod.
   \end{align*}

\begin{thm}[{\cite[Theorem 5.5]{KL4}}]
   \(\mathcal{F}_{\epsilon}\) induces an exact monoidal functor
   \[\mathcal{F}_{\epsilon} \colon \bigoplus_{\beta \in \mathsf{Q}^{+}} R_{\beta} \gmod \to \mathcal{C}(\epsilon).\]
\end{thm}

\subsection{Construction of the bimodules \(V'^{\otimes \beta}_{\mathbb{O}}\)}\label{subsec:Construction of the bimodules 1}
   For the data \(I = \mathbb{Z}\) above, consider the direct sum decomposition \(I = I_{\mathsf{even}} = \mathbb{Z}\), and set \(J = I_{\mathsf{even}} \times \{0,1\} = \mathbb{Z} \times \{0,1\}\).
   Based on these data, we construct the quiver Hecke--Clifford superalgebra \(RC_{\beta}\) of type \(A_{\infty}\).
   Define a map \(X' \colon J \to \Bbbk^{\times}\) by
   \[X'(j,\varepsilon) = q^{-2j}.\]

   For each \(\nu \in J^{\beta}\), the superalgebras \(\mathcal{P}'_{\nu}\) and \(\mathbb{O}'_{\nu}\) are defined by
   \begin{align*}
      \mathcal{P}'_{\nu} &= \langle e(\nu),x_i \mid 1 \leq i \leq n \rangle_{\text{subalg}}, \\
      \mathbb{O}'_{\nu} &= \Bbbk[[X_1 - X(\nu_1),\ldots,X_n - X(\nu_n)]]e(\nu).
   \end{align*}
   Here \(\mathcal{P}'_{\nu}\) is the \(RC_{\beta}\)-analogue of the subalgebra \(\mathcal{P}_{\nu}\) of the polynomial representation of \(R_{\beta}\), and \(\mathbb{O}'_{\nu}\) is its completion.
   In analogy with the case of \(R_{\beta}\), the embedding \(\mathcal{P}'_{\nu} \hookrightarrow \mathbb{O}'_{\nu}\) is given by
   \[\begin{array}{ccc}
      \mathcal{P}'_{\nu} & \hookrightarrow & \mathbb{O}'_{\nu} \\
      \rotatebox{90}{$\in$} & & \rotatebox{90}{$\in$} \\
      x_i e(\nu) & \longmapsto & \frac{1}{x(\nu_i)}(X_i - X(\nu_i))e(\nu) \\
   \end{array}\]
   Let \(\mathbb{K}'_{\nu}\) be the quotient field of \(\mathbb{O}'_{\nu}\), and set
   \begin{align*}
      \mathbb{O}'_{\beta} &= \bigoplus_{\nu \in J^{\beta}} \mathbb{O}'_{\nu}, \\
      \mathbb{K}'_{\beta} &= \bigoplus_{\nu \in J^{\beta}} \mathbb{K}'_{\nu}.
   \end{align*}
   For \(\nu \in J^{\beta}\) and \(1 \leq i \leq n\), set
   \[V_{\nu_{i}} = (\mathcal{W}_{1,\epsilon})_{\text{aff}},\]
   and
   \[V_{\nu} = V_{\nu_1} \otimes \cdots \otimes V_{\nu_n}.\]
   Setting \(K = \Bbbk[X_1^{\pm1},\ldots,X_{n}^{\pm1}]\), we naturally have
   \[V_{\nu} \simeq K \otimes_{\Bbbk} \mathcal{W}_{1,\epsilon}^{\otimes n}.\]
   Define
   \begin{align*}
      V'^{\otimes \beta} &= \bigoplus_{\nu \in J^{\beta}}V_{\nu}e(\nu), \\
      V'^{\otimes \beta}_{\mathbb{O}} &= \mathbb{O}'_{\beta} \otimes_{K} V'^{\otimes \beta}, \\
      V'^{\otimes \beta}_{\mathbb{K}} &= \mathbb{K}'_{\beta} \otimes_{\mathbb{O}'_{\beta}} V'^{\otimes \beta}_{\mathbb{O}} = \mathbb{K}'_{\beta} \otimes_{K} V'^{\otimes \beta}.
   \end{align*}
   Let \(P \colon \mathcal{W}_{1,\epsilon} \to \mathcal{W}_{1,\epsilon}\) be an odd linear map with \(P^2 = 1\).
   We define the map \(-.e(\nu)\c_i \colon  V'^{\otimes \beta}_{\mathbb{K}}e(\nu) \longrightarrow V'^{\otimes \beta}_{\mathbb{K}}e(c_{i}(\nu))\) by
   \[f \otimes (v_1 \otimes \cdots \otimes v_n).e(\nu)\c_i = (-1)^{\Par(v_{i}) + \cdots + \Par(v_n)} \sqrt{-1} f \otimes (\cdots \otimes P(v_i) \otimes \cdots)e(c_{i}(\nu)).\]
   We then define the maps \(-.e(\nu)x_i \colon V'^{\otimes \beta}_{\mathbb{K}}e(\nu) \to V'^{\otimes \beta}_{\mathbb{K}}e(\nu)\) and \(-.e(\nu)\tilde{s}_k \colon V'^{\otimes \beta}_{\mathbb{K}}e(\nu) \to V'^{\otimes \beta}_{\mathbb{K}}e(s_{k}(\nu))\).
   When \(\pi_2(\nu_j) = \overline{0}\) for all \(j\), the maps coincide with the right action of \(R_{\beta}\):
   \begin{align*}
      f \otimes (v_1 \otimes \cdots \otimes v_n).e(\nu)x_i &= f x_i \otimes (v_1 \otimes \cdots \otimes v_n)e(\nu), \\
      f \otimes (v_1 \otimes \cdots \otimes v_n).e(\nu)\tilde{s}_k &= \begin{cases}
         s_{k}(f)P_{\pi_{1}(\nu_k),\pi_{1}(\nu_{k+1})}(x_k,x_{k+1}) & \\
         \otimes (\cdots \otimes R_{1,1}^{\mathrm{norm}}(v_{k} \otimes v_{k+1}) \otimes \cdots)e(s_{k}(\nu)) & \text{if $\nu_k \neq \nu_{k+1},$}\\
         (s_{k}(f) \otimes (\cdots \otimes R_{1,1}^{\mathrm{norm}}(v_{k} \otimes v_{k+1}) \otimes \cdots) & \\
         - f \otimes (v_1 \otimes \cdots \otimes v_n))/(x_k - x_{k+1})e(\nu) & \text{if $\nu_k = \nu_{k+1},$}
      \end{cases}
   \end{align*}
   where
   \[P_{\pi_{1}(\nu_k),\pi_{1}(\nu_{k+1})}(x_k,x_{k+1}) = \begin{cases}
      1 & \text{if $\pi_{1}(\nu_k) > \pi_{1}(\nu_{k+1}),$} \\
      \tilde{\mathcal{Q}}_{\pi_{1}(\nu_k),\pi_{1}(\nu_{k+1})}(x_k,x_{k+1}) & \text{if $\pi_{1}(\nu_k) < \pi_{1}(\nu_{k+1}).$}
   \end{cases}\]
   If there exists \(j\) with \(\pi_2(\nu_j) = \overline{1}\), let \(S_\nu\) denote the set of indices \(j\) satisfying \(\pi_2(\nu_j) = \overline{1}\), indexed by the ordering as
   \[S_{\nu} = \{j_1 < \cdots < j_l\}.\]
   The maps are then defined by
   \begin{align*}
      -.e(\nu)x_i &= -.e(\nu)\c_{j_1} \cdots \c_{j_l}.x_i.(-1)^{\delta_{i,j_l}}\c_{j_l} \cdots (-1)^{\delta_{i,j_1}}\c_{j_1}, \\
      -.e(\nu)\tilde{s}_{j_k} &= -.e(\nu)\c_{j_1} \cdots \c_{j_l}. \tilde{s}_{j_k} .\c_{s_{k}(j_l)} \cdots \c_{s_{k}(j_1)}.
   \end{align*}
   Note that the \(x_i\) and \(\tilde{s}_k\) on the right-hand sides actually denote \(-.e(c_{j_l} \cdots c_{j_1}(\nu))x_i\) and \(-.e(c_{j_l} \cdots c_{j_1}(\nu))\tilde{s}_k\). Since \(S_{c_{j_l} \cdots c_{j_1}(\nu)} = \varnothing\), these actions are already defined.
\begin{lem}\label{lem:construction of right action of untwisted version}
   The maps defined above endow \(V'^{\otimes \beta}_{\mathbb{K}}\) with a right \(\mathcal{KS}_{\beta}\)-module structure.
\end{lem}
\begin{proof}
   That \(e^{\dagger}(\beta)\mathcal{KS}_{\beta}e^{\dagger}(\beta)\) acts is a consequence of the construction of the Schur--Weyl bimodule for \(\mathcal{U}(\epsilon)\) recalled above.
   Combined with a routine calculation, this shows that the action of \(\mathcal{KS}_{\beta}\) is well defined.
\end{proof}

\begin{rem}
   A direct calculation from the definitions yields
   \[f \otimes (v_1 \otimes \cdots \otimes v_n)e(\nu).e(\nu)x_i = (-1)^{\pi_2(\nu_i)}x_{i}f \otimes (v_1 \otimes \cdots \otimes v_n)e(\nu).\]
\end{rem}

\begin{rem}
   Twisting the action by the antiautomorphism of \(\mathcal{KS}_{\beta}\) yields a left action of \(\mathcal{KS}_{\beta}\) on \(V'^{\otimes \beta}_{\mathbb{K}}\).
   Only the action of the odd generators \(\c_i\) changes:
   \[\c_i e(\nu).f \otimes (v_1 \otimes \cdots \otimes v_n)e(\nu)= (-1)^{\Par(v_{1}) + \cdots + \Par(v_{i-1})}f \otimes (\cdots \otimes P(v_i) \otimes \cdots)e(c_{i}(\nu)).\]
\end{rem}

\begin{thm}\label{thm:construction of bimodule of untwisted version}
   \(V'^{\otimes \beta}_{\mathbb{O}}\) is a right \(RC_{\beta}\)-submodule of \(V'^{\otimes \beta}_{\mathbb{K}}\).
\end{thm}
\begin{proof}
   By Theorem \ref{thm:Schur--Weyl subbimod}, \(V'^{\otimes \beta}_{\mathbb{O}}e^{\dagger}(\beta)\) is closed under the action of \(R_{\beta} \simeq e^{\dagger}(\beta)RC_{\beta}e^{\dagger}(\beta)\).
   Since \(P\) is a linear map on \(\mathcal{W}_{1,1}\), the space \(V'^{\otimes \beta}_{\mathbb{O}}\) is closed under the action of \(\c_1,\ldots,\c_n\).
   Every element of \(RC_{\beta}\) acts via compositions involving \(e^{\dagger}(\beta)RC_{\beta}e^{\dagger}(\beta)\) and \(\c_1, \ldots, \c_n\).
   Hence \(V'^{\otimes \beta}_{\mathbb{O}}\) is closed under the action of \(RC_{\beta}\).
\end{proof}

\begin{dfn}\label{dfn:def of Sw bimod for untwisted queer}
   For each \(\beta \in \mathsf{Q}^{+}\), define a subsuperalgebra \(\mathcal{U}_{\beta}(\epsilon)\) of \(\mathcal{U}(\epsilon)\) by
   \[\mathcal{U}_{\beta}(\epsilon) = \{x \in \mathcal{U}(\epsilon) \mid x.(m.a) = (x.m).a~\text{for all}~a \in RC_{\beta}, m \in V'^{\otimes \beta}_{\mathbb{O}}\}.\]
   Let \(\mathcal{U}(\epsilon)^{P}\) denote the subsuperalgebra \(\bigcap_{\beta \in \mathsf{Q^{+}}}\mathcal{U}_{\beta}(\epsilon)\) of \(\mathcal{U}(\epsilon)\). For every \(\beta \in \mathsf{Q}^{+}\), \(V'^{\otimes \beta}_{\mathbb{O}}\) is then a \((\mathcal{U}(\epsilon)^{P}, RC_{\beta})\)-bisupermodule.
   Let \(\mathcal{F}_{\epsilon,\beta}^{P}\) denote the functor \(V'^{\otimes \beta}_{\mathbb{O}} \otimes_{RC_{\beta}} - \colon RC_{\beta} \sgMod \to \mathcal{U}(\epsilon)^{P} \sgMod \),
   and set
   \begin{align*}
      \mathcal{F}_{\epsilon,l}^{P} &= \bigoplus_{\height \beta = l} \mathcal{F}_{\epsilon,\beta}^{P} \colon \bigoplus_{\height \beta = l}RC_{\beta} \sgMod  \to \mathcal{U}(\epsilon)^{P} \sgMod,  \\
      \mathcal{F}_{\epsilon}^{P} &= \bigoplus_{l \geq 0} \mathcal{F}_{\epsilon,l}^{P} \colon \bigoplus_{\beta \in \mathsf{Q}^{+}} RC_{\beta} \sgMod  \to \mathcal{U}(\epsilon)^{P} \sgMod.
   \end{align*}
\end{dfn}

\subsection{Construction of the bimodules \(V^{\mathsf{tw} \otimes \beta}_{\mathbb{O}}\)}\label{subsec:Construction of the bimodules 2}

   For the data \(I = \mathbb{Z}\), consider the direct sum decomposition \(I = I_{\mathsf{even}} = \mathbb{Z}\). Set \(J^{\mathsf{tw}} = J = \mathbb{Z} \times \{0,1\}\), and define a map \(X^{\mathsf{tw}} \colon J^{\mathsf{tw}} \to \Bbbk^{\times}\) by
   \[X^{\mathsf{tw}}(j,\varepsilon) = q^{-(-1)^{\varepsilon}2j}.\]
   This produces the quiver Hecke--Clifford superalgebra \(RC_{\beta}\) of type \(A_{\infty}\).
   For each \(\nu \in J^{\mathsf{tw} \beta}\), the superalgebras \(\mathcal{P}^{\mathsf{tw}}_{\nu}\) and \(\mathbb{O}^{\mathsf{tw}}_{\nu}\) are defined by
   \begin{align*}
      \mathcal{P}^{\mathsf{tw}}_{\nu} &= \langle e(\nu),x_i \mid 1 \leq i \leq n \rangle_{\text{subalg}}, \\
      \mathbb{O}^{\mathsf{tw}}_{\nu} &= \Bbbk[[X_1 - X(\nu_1),\ldots,X_n - X(\nu_n)]]e(\nu).
   \end{align*}
   These are identical to \(\mathcal{P}'_{\nu}\) and \(\mathbb{O}'_{\nu}\).
   In contrast with the case of \(R_{\beta}\), the embedding \(\mathcal{P}^{\mathsf{tw}}_{\nu} \hookrightarrow \mathbb{O}^{\mathsf{tw}}_{\nu}\) (cf. \cite[\(\S 4.5\)]{KKT}) is given by
   \[\begin{array}{ccc}
      \mathcal{P}^{\mathsf{tw}}_{\nu} & \hookrightarrow & \mathbb{O}^{\mathsf{tw}}_{\nu} \\
      \rotatebox{90}{$\in$} & & \rotatebox{90}{$\in$} \\
      x_i e(\nu) & \longmapsto & \frac{X_i - X(\nu_i)}{X_i + X(\nu_i)}e(\nu)\\
   \end{array}\]
   The remaining definitions are analogous.
   Let \(\mathbb{K}^{\mathsf{tw}}_{\nu}\) be the quotient field of \(\mathbb{O}^{\mathsf{tw}}_{\nu}\), and set
   \begin{align*}
      \mathbb{O}^{\mathsf{tw}}_{\beta} &= \bigoplus_{\nu \in J^{\mathsf{tw} \beta}} \mathbb{O}^{\mathsf{tw}}_{\nu},\\
      \mathbb{K}^{\mathsf{tw}}_{\beta} &= \bigoplus_{\nu \in J^{\mathsf{tw} \beta}} \mathbb{K}^{\mathsf{tw}}_{\nu}.
   \end{align*}
   For \(\nu \in J^{\mathsf{tw} \beta}\) and \(1 \leq i \leq n\), set
   \[V_{\nu_{i}} = (\mathcal{W}_{1,\epsilon})_{\text{aff}},\]
   and
   \[V_{\nu} = V_{\nu_1} \otimes \cdots \otimes V_{\nu_n}.\]
   Setting \(K = \Bbbk[X_1^{\pm1},\ldots,X_{n}^{\pm1}]\), we naturally have
   \[V_{\nu} \simeq K \otimes_{\Bbbk} \mathcal{W}_{1,\epsilon}^{\otimes n}.\]
   Define
   \begin{align*}
      V^{\mathsf{tw}\otimes \beta} &= \bigoplus_{\nu \in J^{\mathsf{tw}\beta}}V_{\nu}e(\nu), \\
      V^{\mathsf{tw}\otimes \beta}_{\mathbb{O}} &= \mathbb{O}^{\mathsf{tw}}_{\beta} \otimes_{K} V^{\mathsf{tw} \otimes \beta}, \\
      V^{\mathsf{tw}\otimes \beta}_{\mathbb{K}} &= \mathbb{K}^{\mathsf{tw}}_{\beta} \otimes_{\mathbb{O}^{\mathsf{tw}}_{\beta}} V^{\mathsf{tw} \otimes \beta}_{\mathbb{O}} = \mathbb{K}^{\mathsf{tw}}_{\beta} \otimes_{K} V^{\mathsf{tw} \otimes \beta}.
   \end{align*}

\begin{dfn}
   For \(1 \leq i \leq n\), define a transformation \(T_i\) on \(\mathbb{K}^{\mathsf{tw}}_{\beta}\) by
   \[T_{i}(f(X_1,\ldots,X_n)e(\nu)) = f(X_1,\ldots,X_{i-1},X_{i}^{-1},X_{i+1},\ldots,X_n)e(c_{i}(\nu)).\]
\end{dfn}

\begin{rem}
   Restricting \(T_i\) to \(\mathbb{O}^{\mathsf{tw}}_{\beta}\), we have
   \begin{align*}
      &T_{i}((X_i - X(\nu_i))e(\nu)) \\
      &= (X_i^{-1} - X(\nu_i))e(c_{i}(\nu)) \\
      &= (X_i^{-1} - X((c_{i}(\nu))_i)^{-1})e(c_{i}(\nu)) \\
      &= -\frac{X_i - X((c_{i}(\nu))_i)}{X((c_{i}(\nu))_i) X_{i}}e(c_{i}(\nu)) \\
      &= -\frac{X_i - X((c_{i}(\nu))_i)}{X((c_{i}(\nu))_i)}.\frac{1}{X_i - X((c_{i}(\nu))_i) + X((c_{i}(\nu))_i)}e(c_{i}(\nu)) \\
      &= -\frac{X_i - X((c_{i}(\nu))_i)}{X((c_{i}(\nu))_i)}.(\sum_{j = 0}^{\infty}X((c_{i}(\nu))_i)^{-j-1}(X_i - X((c_{i}(\nu))_i))^j)e(c_{i}(\nu)) \in \mathbb{O}^{\mathsf{tw}}_{c_{i}(\nu)} \subset \mathbb{O}^{\mathsf{tw}}_{\beta},
   \end{align*}
   so that \(T_i\) is well defined as a transformation on \(\mathbb{O}^{\mathsf{tw}}_{\beta}\).
   The same calculation shows that \(T_i\) is well defined on \(\mathbb{K}_{\beta}^{\mathsf{tw}}\) as well.
\end{rem}
   Let \(P \colon \mathcal{W}_{1,\epsilon} \to \mathcal{W}_{1,\epsilon}\) be an odd involution.
   The map \(-.e(\nu)\c_i \colon  V^{\mathsf{tw} \otimes \beta}_{\mathbb{K}}e(\nu) \longrightarrow V^{\mathsf{tw} \otimes \beta}_{\mathbb{K}}e(c_{i}(\nu))\) is given by
   \[f \otimes (v_1 \otimes \cdots \otimes v_n).e(\nu)\c_i = (-1)^{\Par(v_{i}) + \cdots + \Par(v_n)} \sqrt{-1} T_i(fe(\nu)) \otimes (\cdots \otimes P(v_i) \otimes \cdots)e(c_{i}(\nu)).\]
   The maps \(-.e(\nu)x_i \colon V^{\mathsf{tw} \otimes \beta}_{\mathbb{K}}e(\nu) \to V^{\mathsf{tw} \otimes \beta}_{\mathbb{K}}e(\nu)\) and \(-.e(\nu)\tilde{s}_k \colon V^{\mathsf{tw} \otimes \beta}_{\mathbb{K}}e(\nu) \to V^{\mathsf{tw} \otimes \beta}_{\mathbb{K}}e(s_{k}(\nu))\) are defined in analogy with Lemma \ref{lem:construction of right action of untwisted version}.
   When \(\pi_2(\nu_j) = \overline{0}\) for all \(j\), the maps are given by
   \begin{align*}
      f \otimes (v_1 \otimes \cdots \otimes v_n).e(\nu)x_i &= f x_i \otimes (v_1 \otimes \cdots \otimes v_n)e(\nu), \\
      f \otimes (v_1 \otimes \cdots \otimes v_n).e(\nu)\tilde{s}_k &= \begin{cases}
         s_{k}(f)P_{\pi_{1}(\nu_k),\pi_{1}(\nu_{k+1})}(x_k,x_{k+1}) & \\
         \otimes (\cdots \otimes R_{1,1}^{\mathrm{norm}}(v_{k} \otimes v_{k+1}) \otimes \cdots)e(s_{k}(\nu)) & \text{if $\nu_k \neq \nu_{k+1},$}\\
         (s_{k}(f) \otimes (\cdots \otimes R_{1,1}^{\mathrm{norm}}(v_{k} \otimes v_{k+1}) \otimes \cdots) & \\
         - f \otimes (v_1 \otimes \cdots \otimes v_n))/(x_k - x_{k+1})e(\nu) & \text{if $\nu_k = \nu_{k+1}.$}
      \end{cases}
   \end{align*}
   If there exists \(j\) with \(\pi_2(\nu_j) = \overline{1}\), let \(S_\nu\) denote the set of indices \(j\) satisfying \(\pi_2(\nu_j) = \overline{1}\), indexed by the ordering as
   \[S_{\nu} = \{j_1 < \cdots < j_l\}.\]
   The maps are then defined by
   \begin{align*}
      -.e(\nu)x_i &= -.e(\nu)\c_{j_1} \cdots \c_{j_l}.x_i.(-1)^{\delta_{i,j_l}}\c_{j_l} \cdots (-1)^{\delta_{i,j_1}}\c_{j_1}, \\
      -.e(\nu)\tilde{s}_{j_k} &= -.e(\nu)\c_{j_1} \cdots \c_{j_l}. \tilde{s}_{j_k} .\c_{s_{k}(j_l)} \cdots \c_{s_{k}(j_1)}.
   \end{align*}
\begin{lem}
   The maps defined above endow \(V^{\mathsf{tw} \otimes \beta}_{\mathbb{K}}\) with a right \(\mathcal{KS}_{\beta}\)-module structure.
\end{lem}
The claim follows by a direct verification of the relations of \(\mathcal{KS}_{\beta}\).

\begin{rem}
   The relation between \(T_i\) and \(x_{i}e(\nu)\) reads
   \begin{align*}
      T_{i}(x_{i}e(\nu)) &= T_{i}(\frac{X_i - X(\nu_i)}{X_i + X(\nu_i)}e(\nu)) \\
      &= \frac{X_i^{-1} - X(\nu_i)}{X_i^{-1}+ X(\nu_i)}e(c_{i}(\nu)) \\
      &= \frac{X_i^{-1} - X((c_{i}(\nu))_i)^{-1}}{X_i^{-1}+ X((c_{i}(\nu))_i)^{-1}}e(c_{i}(\nu)) \\
      &= -\frac{X_i - X((c_{i}(\nu))_i)}{X_i+ X((c_{i}(\nu))_i)}e(c_{i}(\nu)) \\
      &= -x_{i}e(c_i(\nu)).
   \end{align*}
   Consequently,
   \[f \otimes (v_1 \otimes \cdots \otimes v_n)e(\nu).e(\nu)x_i = (-1)^{\pi_2(\nu_i)}x_{i}f \otimes (v_1 \otimes \cdots \otimes v_n)e(\nu)\]
   for every \(\nu\).
\end{rem}

\begin{rem}
   Twisting the action by the antiautomorphism of \(\mathcal{KS}_{\beta}\) yields a left action of \(\mathcal{KS}_{\beta}\) on \(V'^{\otimes \beta}_{\mathbb{K}}\).
   Only the action of the odd generators \(\c_i\) changes:
   \[\c_i e(\nu).f \otimes (v_1 \otimes \cdots \otimes v_n)e(\nu)= (-1)^{\Par(v_{1}) + \cdots + \Par(v_{i-1})} T_{i}(fe(\nu)) \otimes (v_1 \otimes \cdots \otimes P(v_i) \otimes \cdots \otimes v_n).\]
\end{rem}

\begin{thm}\label{thm:construction of bimodule of twisted version}
   \(V^{\mathsf{tw} \otimes \beta}_{\mathbb{O}}\) is a right \(RC_{\beta}\)-submodule of \(V^{\mathsf{tw} \otimes \beta}_{\mathbb{K}}\).
\end{thm}
\begin{proof}
   The proof is analogous to that of Theorem \ref{thm:construction of bimodule of untwisted version}.
\end{proof}

\begin{dfn}\label{dfn:def of Sw bimod for twisted queer}
   For each \(\beta \in \mathsf{Q}^{+}\), define a subsuperalgebra \(\mathcal{U}_{\beta}^{\mathsf{tw}}(\epsilon)\) of \(\mathcal{U}(\epsilon)\) by
   \[\mathcal{U}_{\beta}^{\mathsf{tw}}(\epsilon) = \{x \in \mathcal{U}(\epsilon) \mid x.(m.a) = (x.m).a~\text{for all}~a \in RC_{\beta}, m \in V^{\mathsf{tw} \otimes \beta}_{\mathbb{O}}\}.\]
   Let \(\mathcal{U}^{\mathsf{tw}}(\epsilon)^{P}\) denote the subsuperalgebra \(\bigcap_{\beta \in \mathsf{Q^{+}}}\mathcal{U}_{\beta}^{\mathsf{tw}}(\epsilon)\) of \(\mathcal{U}(\epsilon)\). For every \(\beta \in \mathsf{Q}^{+}\), \(V^{\mathsf{tw} \otimes \beta}_{\mathbb{O}}\) is then a \((\mathcal{U}^{\mathsf{tw}}(\epsilon)^{P}, RC_{\beta})\)-bisupermodule.
   Let \(\mathcal{F}_{\epsilon,\beta}^{\mathsf{tw} P}\) denote the functor \(V^{\mathsf{tw} \otimes \beta}_{\mathbb{O}} \otimes_{RC_{\beta}} - \colon RC_{\beta} \sgMod \to \mathcal{U}^{\mathsf{tw}}(\epsilon)^{P} \sgMod \),
   and set
   \begin{align*}
      \mathcal{F}_{\epsilon,n}^{\mathsf{tw} P} &= \bigoplus_{\height \beta = n} \mathcal{F}_{\epsilon,\beta}^{\mathsf{tw} P} \colon \bigoplus_{\height \beta = n}RC_{\beta} \sgMod  \to \mathcal{U}^{\mathsf{tw}}(\epsilon)^{P} \sgMod,  \\
      \mathcal{F}_{\epsilon}^{\mathsf{tw} P} &= \bigoplus_{n \geq 0} \mathcal{F}_{\epsilon,n}^{\mathsf{tw} P} \colon \bigoplus_{\beta \in \mathsf{Q}^{+}} RC_{\beta} \sgMod  \to \mathcal{U}^{\mathsf{tw}}(\epsilon)^{P} \sgMod.
   \end{align*}
\end{dfn}

\subsection{Properties of the Schur--Weyl functors}\label{subsec:Properties of the Schur--Weyl functors}

\begin{prop}[cf. {\cite[Theorem 3.4]{KKK}}]
   The functors \(\mathcal{F}_{\epsilon,\beta}^{P}\) and \(\mathcal{F}_{\epsilon,\beta}^{\mathsf{tw} P}\) restrict to functors
   \begin{align*}
      \mathcal{F}_{\epsilon,\beta}^{P} &\colon RC_{\beta} \sgmod  \to \mathcal{U}(\epsilon)^{P} \sgmod,  \\
      \mathcal{F}_{\epsilon,\beta}^{\mathsf{tw} P} &\colon RC_{\beta} \sgmod  \to \mathcal{U}^{\mathsf{tw}}(\epsilon)^{P} \sgmod;
   \end{align*}
   that is, both \(\mathcal{F}_{\epsilon,\beta}^{P}\) and \(\mathcal{F}_{\epsilon,\beta}^{\mathsf{tw} P}\) send finite-dimensional modules to finite-dimensional modules.
\end{prop}
\begin{proof}
   We treat the case of \(\mathcal{F}_{\epsilon,\beta}^{P}\).
   Define the superalgebra \(\mathcal{P}'_{\beta}\) by
   \begin{align*}
      \mathcal{P}'_{\beta} =  \langle e(\nu), x_i \in RC_{\beta} \mid \nu \in J^{\beta},~ 1 \leq i \leq n \rangle  \hookrightarrow \mathbb{O}'_{\beta}.
   \end{align*}
   For \(W \in RC_{\beta} \sgmod \), the surjection
   \[V_{\mathbb{O}}'^{\otimes \beta} \otimes_{\mathcal{P}'_{\beta}} W \twoheadrightarrow V_{\mathbb{O}}'^{\otimes \beta} \otimes_{RC_{\beta}} W = \mathcal{F}_{\epsilon,\beta}^{P}(W)\]
   reduces the problem to showing that \(V_{\mathbb{O}}'^{\otimes \beta} \otimes_{\mathcal{P}'_{\beta}} W\) is finite-dimensional.
   This can be rewritten as
   \begin{align*}
      V_{\mathbb{O}}'^{\otimes \beta} \otimes_{\mathcal{P}'_{\beta}} W &\simeq \bigoplus_{\nu \in J^{\beta}} (\mathbb{O}'_{\nu} \otimes_{K} K \otimes_{\Bbbk} \mathcal{W}_{1,\epsilon}^{\otimes n}) \otimes_{\mathcal{P}'_{\nu}} W \\
      &\simeq \bigoplus_{\nu \in J^{\beta}} \mathcal{W}_{1,\epsilon}^{\otimes n} \otimes_{\Bbbk} \mathbb{O}'_{\nu} \otimes_{\mathcal{P}'_{\nu}} W.
   \end{align*}
   Since \(W\) is \(\mathbb{Z}\)-graded and finite-dimensional, \(W_k = 0\) for sufficiently large \(k \in \mathbb{Z}\).
   As \(x_i e(\nu) \in \mathcal{P}'_{\nu}\) has \(\mathbb{Z}\)-degree \(2\), every polynomial \(f \in \Bbbk[x_1,\ldots,x_n] \subset \mathcal{P}'_{\beta}\) of sufficiently high degree satisfies
   \[f.V = 0.\]
   Likewise, for sufficiently large \(k_i \in \mathbb{Z}\),
   \begin{align*}
      \bigoplus_{\nu \in J^{\beta}} \mathcal{W}_{1,\epsilon}^{\otimes n} \otimes_{\Bbbk} \mathbb{O}'_{\nu} \otimes_{\mathcal{P}'_{\nu}} W &\ni (v_1 \otimes \cdots \otimes v_n) \otimes (X_i - X(\nu_i))^{k_i}e(\nu) \otimes w \\
      &= (v_1 \otimes \cdots \otimes v_n) \otimes e(\nu) \otimes X(\nu_i)^{k_i} x_i^{k_i} e(\nu).w \\
      &= 0.
   \end{align*}
   Hence \(\mathcal{F}_{\epsilon,\beta}^{P}(W)\) is finite-dimensional.
   The case of \(\mathcal{F}_{\epsilon,\beta}^{\mathsf{tw} P}\) proceeds similarly: for sufficiently large \(k_i\),
   \begin{align*}
      \bigoplus_{\nu \in J^{\mathsf{tw} \beta}} \mathcal{W}_{1,\epsilon}^{\otimes n} \otimes_{\Bbbk} \mathbb{O}_{\nu}^{\mathsf{tw}} \otimes_{\mathcal{P}_{\nu}^{\mathsf{tw}}} W &\ni (v_1 \otimes \cdots \otimes v_n) \otimes (X_i - X(\nu_i))^{k_i}e(\nu) \otimes w \\
      &= (v_1 \otimes \cdots \otimes v_n) \otimes (X_i + X(\nu_i))^{N_i}e(\nu) \otimes  x_i^{k_i} e(\nu).w \\
      &= 0.
   \end{align*}
   Hence \(\mathcal{F}_{\epsilon,\beta}^{\mathsf{tw} P}(W)\) is finite-dimensional.
\end{proof}
\begin{lem}[{\cite[Proposition 3.7]{KKK}}]\label{lem:flatness lemma}
   Let \(A \to B\) be an algebra homomorphism satisfying:
   \begin{enumerate}[(a)]
      \item \(B\) is finitely generated and projective as an \(A\)-module;
      \item \(\Hom_{A}(B,A)\) is projective as a \(B\)-module;
      \item the global dimension of \(B\) is finite.
   \end{enumerate}
   Then any \(B\)-module that is flat as an \(A\)-module is also flat as a \(B\)-module.
\end{lem}

\begin{thm}\label{thm:exactness of SW functor for queer}
   The functors \(\mathcal{F}_{\epsilon,\beta}^{P}\) and \(\mathcal{F}_{\epsilon,\beta}^{\mathsf{tw} P}\) are exact.
\end{thm}
\begin{proof}
   We establish the claim for \(\mathcal{F}_{\epsilon}^{P}\).
   Consider the subsuperalgebras \(\mathcal{P}'_{\beta}\) and \(\mathcal{A}'_{\beta} = \langle e(\nu), x_i, \c_i \mid \nu \in J^{\beta}, 1 \leq i \leq n \rangle\) of \(RC_{\beta}\).
   We verify conditions (a), (b), and (c) of Lemma \ref{lem:flatness lemma} for each of the inclusions \(\mathcal{P}'_{\beta} \hookrightarrow \mathcal{A}'_{\beta} \hookrightarrow RC_{\beta}\).
   \begin{enumerate}[(i)]
      \item Case \(\mathcal{P}'_{\beta} \hookrightarrow \mathcal{A}'_{\beta}\):
      \begin{enumerate}[(a)]
         \item Since
         \[\{x_{1}^{a_1} \cdots x_{n}^{a_n}e(\nu)\c_{1}^{\eta_1} \cdots \c_{n}^{\eta_n} \mid a_i \in \mathbb{Z}_{\geq 0},~\eta_i \in \mathbb{Z}/2\mathbb{Z},~\nu \in J^{\beta}\}\]
         is a basis of \(\mathcal{A}'_{\beta}\), the \(\mathcal{P}'_{\beta}\)-module decomposition
         \[\mathcal{A}'_{\beta} = \bigoplus_{\eta \in \mathbb{Z}/2\mathbb{Z}^n} \mathcal{P}'_{\beta}\c_{1}^{\eta_1} \cdots \c_{n}^{\eta_n}\]
         holds. Hence \(\mathcal{A}'_{\beta}\) is finitely generated and projective over \(\mathcal{P}'_{\beta}\).
         \item From the chain of \(\mathcal{P}'_{\beta}\)-module isomorphisms
         \begin{align*}
            \HOM_{\mathcal{P}'_{\beta}}(\mathcal{A}'_{\beta},\mathcal{P}'_{\beta}) &\simeq \HOM_{\mathcal{P}'_{\beta}}(\mathcal{P}'_{\beta} \otimes_{\Bbbk} \mathcal{C}_{n},\mathcal{P}'_{\beta}) \\
            &\simeq \HOM_{\mathcal{P}'_{\beta}}(\mathcal{P}'_{\beta} , \Hom_{\Bbbk}(\mathcal{C}_{n},\mathcal{P}'_{\beta})) \\
            &\simeq \Hom_{\Bbbk}(\mathcal{C}_{n},\mathcal{P}'_{\beta}) \\
            &\simeq \mathcal{C}_{n}^{*} \otimes \mathcal{P}'_{\beta} \\
            &\simeq \mathcal{C}_{n} \otimes \mathcal{P}'_{\beta} \simeq \mathcal{A}'_{n},
         \end{align*}
         \(\HOM_{\mathcal{P}'_{\beta}}(\mathcal{A}'_{\beta},\mathcal{P}'_{\beta})\) is projective as an \(\mathcal{A}'_{\beta}\)-module.
         \item The idempotent
         \[e^{\dagger}(\beta) = \sum_{i,\pi_2(\nu_i) = 0} e(\nu) \in \mathcal{A}_{\beta}\]
         is a full even idempotent of \(\mathcal{A}'_{\beta}\); by Example \ref{ex:full even idempotent}, \(\mathcal{A}'_{\beta}\) and \(e^{\dagger}(\beta)\mathcal{A}'_{\beta}e^{\dagger}(\beta)\) are therefore Morita superequivalent.
         By the basis theorem for \(\mathcal{A}'_{\beta}\),
         \[e^{\dagger}(\beta)\mathcal{A}'_{\beta}e^{\dagger}(\beta) = \mathcal{P}'_{\beta}.\]
         Since \(\mathcal{P}'_{\beta}\) is a finite direct product of polynomial rings in finitely many variables, its global dimension is finite by Hilbert's syzygy theorem.
         Consequently, the global dimension of \(\mathcal{A}'_{\beta}\) is finite.
      \end{enumerate}
      \item Case \(\mathcal{A}'_{\beta} \hookrightarrow RC_{\beta}\):
      \begin{enumerate}[(a)]
         \item By Remark \ref{rem:PBW for quiver Hecke--Clifford},
         \[\{x_1^{a_1} \cdots x_{n}^{a_n}\c_{1}^{\eta_1} \cdots \c_{n}^{\eta_n}e(\nu)\sigma_{w} \mid a_i \in \mathbb{Z}_{\geq 0},~\eta_i \in \mathbb{Z}/2\mathbb{Z},~\nu \in J^{\beta},~w \in \mathfrak{S}_{n}\}\]
         forms a basis of \(RC_{\beta}\), giving the \(\mathcal{A}'_{\beta}\)-module decomposition
         \[RC_{\beta} = \bigoplus_{w \in \mathfrak{S}_{n}}\mathcal{A}'_{\beta}\sigma_{w}.\]
         Hence \(RC_{\beta}\) is finitely generated and projective over \(\mathcal{A}'_{\beta}\).
         \item For \(\nu \in J^{\beta}\), define the central idempotent \(e'(\nu)\) of \(\mathcal{A}'_{\beta}\) by
         \[e'(\nu) = \sum_{\pi_1(\mu) = \nu} e(\mu),\]
         where \(\pi_1(\mu) = (\pi_1(\mu_1),\ldots,\pi_1(\mu_n)) \in J^{\beta}\).
         The identity
         \[1 = \sum_{\nu \in J^{\beta}}e'(\nu) \in \mathcal{A}'_{\beta}\]
         yields a central idempotent decomposition of \(\mathcal{A}'_{\beta}\); accordingly, any \(\mathcal{A}'_{\beta}\)-module \(M\) decomposes as
         \[M = \bigoplus_{\nu \in J^{\beta}}e'(\nu)M.\]
         This yields the following isomorphisms of \(\mathcal{A}'_{\beta}\)-modules:
         \begin{align*}
            \HOM_{\mathcal{A}'_{\beta}}(RC_{\beta},\mathcal{A}'_{\beta}) &= \bigoplus_{\nu \in J^{\beta}} \HOM_{\mathcal{A}'_{\beta}}(e'(\nu)RC_{\beta},e'(\nu)\mathcal{A}'_{\beta}) \\
            &= \bigoplus_{\nu \in J^{\beta}} \HOM_{\mathcal{A}'_{\beta} e'(\nu)}(e'(\nu)RC_{\beta},e'(\nu)\mathcal{A}'_{\beta}).
         \end{align*}
         The block \(\mathcal{A}'_{\beta}e'(\nu) = e'(\nu)\mathcal{A}'_{\beta}\) associated with \(e'(\nu)\) is isomorphic, as a superalgebra, to \(RC_{\alpha_{\nu_1}} \otimes \cdots \otimes RC_{\alpha_{\nu_n}}\).
         Repeated application of Frobenius reciprocity (Proposition \ref{prop:hom-tensor duality for quiver Hecke--Clifford superalgebra}) together with Lemma \ref{lem:ind and coind are isom} then yields the following isomorphisms (ignoring \(\mathbb{Z}\)-degree):
         \begin{align*}
            &\HOM_{\mathcal{A}'_{\beta} e'(\nu)}(e'(\nu)RC_{\beta},e'(\nu)\mathcal{A}'_{\beta}) \\
            &= \HOM_{RC_{\alpha_{\nu_1}} \otimes \cdots \otimes RC_{\alpha_{\nu_l}}}(\Res_{\nu_1,\ldots,\nu_n}^{\beta} RC_{\beta}, RC_{\alpha_{\nu_1}} \otimes \cdots \otimes RC_{\alpha_{\nu_l}}) \\
            &\simeq \HOM_{RC_{\beta}}(RC_{\beta},\Coind_{\nu_1,\ldots,\nu_n}^{\beta} RC_{\alpha_{\nu_1}} \otimes \cdots \otimes RC_{\alpha_{\nu_l}}) \\
            &\simeq \HOM_{RC_{\beta}}(RC_{\beta},\Ind_{\nu_n,\ldots,\nu_1}^{\beta} RC_{\alpha_{\nu_n}} \otimes \cdots \otimes RC_{\alpha_{\nu_1}}) \\
            &\simeq RC_{\alpha_{\nu_n}} \circ \cdots \circ RC_{\alpha_{\nu_1}}.
         \end{align*}
         Here we have used the fact that \(RC_{\beta} \in RC_{\beta} \sgMod_{\mathsf{L}}\), which is a consequence of Remark \ref{rem:PBW for quiver Hecke--Clifford}.
         Therefore,
         \begin{align*}
            \HOM_{\mathcal{A}'_{\beta} e'(\nu)}(e'(\nu)RC_{\beta},e'(\nu)\mathcal{A}'_{\beta})
            &\simeq \bigoplus_{\nu \in J^{\beta}} RC_{\alpha_{\nu_n}} \circ \cdots \circ RC_{\alpha_{\nu_1}} \\
            &\simeq \bigoplus_{\nu \in J^{\beta}} RC_{\beta}e(\nu_n,\ldots,\nu_1) \\
            &\simeq RC_{\beta}.
         \end{align*}
         Hence \(\Hom_{\mathcal{A}'_{\beta}}(RC_{\beta},\mathcal{A}'_{\beta})\) is projective as an \(RC_{\beta}\)-module.
         \item
         By the Morita equivalence with \(R_{\beta}\) (Theorem \ref{thm:super morita equivariant for quiver Hecke--Clifford superalgebra}), it suffices to verify that the global dimension of \(R_{\beta}\) is finite, which is established in \cite[Corollary 2.9]{K2}, \cite[Theorem 4.6]{KL1}, etc.
      \end{enumerate}
   \end{enumerate}
   Lemma \ref{lem:flatness lemma} now implies that any right \(RC_{\beta}\)-module that is flat as a right \(\mathcal{P}'_{\beta}\)-module is also flat as a right \(RC_{\beta}\)-module.
   Since \(\mathbb{O}'_{\beta}\) is flat as a right \(\mathcal{P}'_{\beta}\)-module by virtue of the flatness of the completion of a Noetherian ring,
   \[V_{\mathbb{O}}'^{\otimes \beta} = \mathbb{O}'_{\beta} \otimes \mathcal{W}_{1,\epsilon}^{\otimes n} \simeq \mathbb{O}_{\beta}'^{\oplus (M+N-1)^n}\]
   is flat as a right \(\mathcal{P}'_{\beta}\)-module as well.
   Hence \(V_{\mathbb{O}}'^{\otimes \beta}\) is flat as a right \(RC_{\beta}\)-module, which entails the exactness of \(\mathcal{F}_{\epsilon}^{P}\).
   The same conclusion holds for \(\mathcal{F}_{\epsilon}^{\mathsf{tw} P}\).
\end{proof}

\begin{ex}\label{ex:segment module}
   The length-\(1\) segment representation \(L(a) = \mathcal{C}_1.e(a,0) \in RC_{\alpha_a} \sgmod ~(a \in \mathbb{Z})\) of the quiver Hecke--Clifford superalgebra of type \(A_{\infty}\) is defined by
   \begin{align*}
      x_i.\c^j e(a,0) &= 0,~\sigma_k.\c^j e(a,0) = 0,~\c_1.\c^j e(a,0) = c^{j+1} e(a,0), \\
      e(a,\varepsilon).\c^j e(a,0) &= \delta_{\varepsilon + j,0}\c^j e(a,b).
   \end{align*}
   Its image under the Schur--Weyl functor is computed below.
   Setting \(\beta = \alpha_a\), we have
   \begin{align*}
      V_{\mathbb{O}}'^{\otimes \beta} &= \Bbbk[[z - q^{-2a}]]\otimes_{\Bbbk[z^{\pm 1}]}(\mathcal{W}_{1,\epsilon})_{\text{aff}}.e(a,0) \oplus \Bbbk[[z - q^{-2a}]]\otimes_{\Bbbk[z^{\pm 1}]}(\mathcal{W}_{1,\epsilon})_{\text{aff}}.e(a,1), \\
      V_{\mathbb{O}}^{\mathsf{tw} \otimes \beta} &= \Bbbk[[z - q^{-2a}]]\otimes_{\Bbbk[z^{\pm 1}]}(\mathcal{W}_{1,\epsilon})_{\text{aff}}.e(a,0) \oplus \Bbbk[[z - q^{2a}]]\otimes_{\Bbbk[z^{\pm 1}]}(\mathcal{W}_{1,\epsilon})_{\text{aff}}.e(a,1).
   \end{align*}
   Since \(RC_{\beta} = \langle x, \c, e(a,0), e(a,1) \rangle\), defining the subalgebra \(C_{\beta}\) by
   \[C_{\beta} = \langle \c,e(a,0),e(a,1)\rangle,\]
   we have \(L(a) = C_{\beta}.e(a,0)\), and the following isomorphisms of super vector spaces hold:
   \begin{align*}
      &\mathcal{F}_{\epsilon,\beta}^{P}(L(a)) \\
      &= (\Bbbk[[z - q^{-2a}]]\otimes_{\Bbbk[z^{\pm 1}]}(\mathcal{W}_{1,\epsilon})_{\text{aff}}.e(a,0) \\
      &\oplus \Bbbk[[z - q^{-2a}]]\otimes_{\Bbbk[z^{\pm 1}]} (\mathcal{W}_{1,\epsilon})_{\text{aff}}.e(a,1)) \otimes_{RC_{\beta}} L(a) \\
      &\simeq (\mathcal{W}_{1,\epsilon}.e(a,0) \oplus \mathcal{W}_{1,\epsilon}.e(a,1)) \otimes_{C_{\beta}} C_{\beta}.e(a,0) \\
      &\simeq \mathcal{W}_{1,\epsilon}, \\
      &\mathcal{F}_{\epsilon,\beta}^{\mathsf{tw} P}(L(a)) \\
      &= (\Bbbk[[z - q^{-2a}]]\otimes_{\Bbbk[z^{\pm 1}]}(\mathcal{W}_{1,\epsilon})_{\text{aff}}.e(a,0) \\
      &\oplus \Bbbk[[z - q^{2a}]]\otimes_{\Bbbk[z^{\pm 1}]}(\mathcal{W}_{1,\epsilon})_{\text{aff}}.e(a,1)) \otimes_{RC_{\beta}} L(a) \\
      &\simeq (\mathcal{W}_{1,\epsilon}.e(a,0) \oplus \mathcal{W}_{1,\epsilon}.e(a,1)) \otimes_{C_{\beta}} C_{\beta}.e(a,0)\\
      &\simeq \mathcal{W}_{1,\epsilon}.
   \end{align*}
\end{ex}

\subsection*{Acknowledgements}
The author is deeply grateful to his supervisor, Hironori Oya, for his devoted and sustained support and for numerous invaluable comments.
The author also wishes to thank Katsuyuki Naoi for providing various lectures on quiver Hecke algebras, the insights from which are reflected in the content of Section \ref{sec:schur--weyl duality via quiver hecke--clifford superalgebra}.
The author is further indebted to Hideya Watanabe for helpful comments on this work.


\begin{thebibliography}{99}
   \bibitem{BK}Brundan, J.; Kleshchev, A.:
   {\em Blocks of cyclotomic Hecke algebras and Khovanov-Lauda algebras}.
   Invent. Math. \textbf{178} (2009), no. 3, 451--484.
   \bibitem{BE}Brundan, J.; Ellis, A. P.:
   {\em Monoidal supercategories}.
   Comm. Math. Phys. \textbf{351} (2017), no. 3, 1045--1089.
   \bibitem{CG}Chen, H.; Guay, N.:
   {\em Twisted affine Lie superalgebra of type $Q$ and quantization of its enveloping superalgebra}.
   Math. Z. \textbf{272} (2012), no. 1--2, 317--347.
   \bibitem{CW}Cheng, S.-J.; Wang, W.:
   {\em Dualities and representations of Lie superalgebras}.
   Grad. Stud. Math., 144
   American Mathematical Society, Providence, RI, 2012. xviii+302 pp.
   \bibitem{GS}Gorelik, Maria; Serganova, Vera:
   {\em On representations of the affine superalgebra $\mathfrak{q}(n)_2$}.
   Mosc. Math. J. \textbf{8} (2008), no. 1, 91--109.
   \bibitem{GJKK}Grantcharov, D.; Jung, J. H.; Kang, S.-J.; Kim, M.:
   {\em Highest weight modules over quantum queer superalgebra $U_{q}(\mathfrak{q}_n)$}.
   Comm. Math. Phys. \textbf{296} (2010), no. 3, 827--860.
   \bibitem{HS}Hoyt, C.; Serganova, V.:
   {\em Classification of finite-growth general Kac-Moody superalgebras}.
   Comm. Algebra \textbf{35} (2007), no. 3, 851--874.
   \bibitem{K1}Kac, V. G.:
   {\em Lie superalgebras}.
   Advances in Math. \textbf{26} (1977), no. 1, 8--96.
   \bibitem{KKK}Kang, S.-J.; Kashiwara, M.; Kim, M.:
   {\em Symmetric quiver Hecke algebras and R-matrices of quantum affine algebras}.
   Invent. Math. \textbf{211} (2018), no. 2, 591--685.
   \bibitem{KKKO1}Kang, S.-J.; Kashiwara, M.; Kim, M.; Oh, S.-J.:
   {\em Simplicity of heads and socles of tensor products}.
   Compos. Math. \textbf{151} (2015), no. 2, 377--396.
   \bibitem{KKKO2}Kang, S.-J.; Kashiwara, M.; Kim, M.; Oh, S.-J.:
   {\em Monoidal categorification of cluster algebras}.
   J. Amer. Math. Soc. \textbf{31} (2018), no. 2, 349--426.
   \bibitem{KKO1}Kang, S.-J.; Kashiwara, M.; Oh, S.-j.:
   {\em Supercategorification of quantum Kac-Moody algebras}.
   Adv. Math. \textbf{242} (2013), 116--162.
   \bibitem{KKO2}Kang, S.-J.; Kashiwara, M.; Oh, S.-j.:
   {\em Supercategorification of quantum Kac-Moody algebras II}.
   Adv. Math. \textbf{265} (2014), 169--240.
   \bibitem{KKT}Kang, S.-J.; Kashiwara, M.; Tsuchioka, S.:
   {\em Quiver Hecke superalgebras}.
   J. Reine Angew. Math. \textbf{711} (2016), 1--54.
   \bibitem{KP}Kashiwara, M.; Park, E.:
   {\em Affinizations and R-matrices for quiver Hecke algebras}.
   J. Eur. Math. Soc. (JEMS) \textbf{20} (2018), no. 5, 1161--1193.
   \bibitem{K2}Kato, S.:
   {\em Poincaré--Birkhoff--Witt bases and Khovanov--Lauda--Rouquier algebras}.
   Duke Math. J. \textbf{163} (2014), no. 3, 619--663.
   \bibitem{KL1}Khovanov, M.; Lauda, A. D.:
   {\em A diagrammatic approach to categorification of quantum groups I}.
   Represent. Theory \textbf{13} (2009), 309--347.
   \bibitem{KL2}Khovanov, M.; Lauda, A. D.:
   {\em A diagrammatic approach to categorification of quantum groups II.}
   Trans. Amer. Math. Soc. \textbf{363} (2011), no. 5, 2685--2700.
   \bibitem{K2}Kleshchev, A.:
   {\em Linear and projective representations of symmetric groups}.
   Cambridge Tracts in Math., 163.
   Cambridge University Press, Cambridge, 2005. xiv+277 pp.
   \bibitem{KL3}Kleshchev, A. S.; Livesey, M.:
   {\em RoCK blocks for double covers of symmetric groups and quiver Hecke superalgebras}.
   Mem. Amer. Math. Soc. \textbf{309} (2025), no. 1564, v+182 pp.
   \bibitem{KOS}Kuniba, A.; Okado, M.; Sergeev, S.:
   {\em Tetrahedron equation and generalized quantum groups}.
   J. Phys. A \textbf{48} (2015), no. 30, 304001, 38 pp.
   \bibitem{KL4}Kwon, J.-H., Lee, S.-M.:
   {\em Super duality for quantum affine algebras of type A}.
   Int. Math. Res. Not. IMRN 2022, no. 23, 18446--18525.
   \bibitem{KY}Kwon, J.-H.; Yu, J.:
   {\em R matrix for generalized quantum group of type A}.
   J. Algebra \textbf{566} (2021), 309--341.
   \bibitem{LV}Lauda, A. D.; Vazirani, M.:
   {\em Crystals from categorified quantum groups}.
   Adv. Math. \textbf{228} (2011), no. 2, 803--861.
   \bibitem{L}Lenzen, F.:
   {\em Clifford-symmetric polynomials}.
   Comm. Algebra \textbf{51} (2023), no. 9, 3981--4011.
   \bibitem{LMZ}Liu, M.; Molev, A.; Zhang, J.:
   {\em Central elements and evaluation map for the quantum queer superalgebras}.
   J. Math. Phys. \textbf{67} (2026), no. 4, Paper No. 041701.
   \bibitem{M1}Mathieu, O.:
   {\em Classification of simple graded Lie algebras of finite growth}.
   Invent. Math. \textbf{108} (1992), no. 3, 455--519.
   \bibitem{M2}McNamara, P. J.:
   {\em Representations of Khovanov-Lauda-Rouquier algebras III: symmetric affine type}.
   Math. Z. \textbf{287} (2017), no. 1-2, 243--286.
   \bibitem{O}Olshanski, G. I.:
   {\em Quantized universal enveloping superalgebra of type $Q$ and a super-extension of the Hecke algebra}.
   Lett. Math. Phys. \textbf{24} (1992), no. 2, 93--102.
   \bibitem{R}Rouquier, R.:
   {\em 2-Kac-Moody algebras},   arXiv:0812.5023v1.
   \bibitem{S}Sergeev, A. N.:
   {\em Tensor algebra of the identity representation as a module over the Lie superalgebras $\mathfrak{Gl}(n,m)$ and $Q(n)$}.
   Mat. Sb. (N.S.) \textbf{123(165)} (1984), no. 3, 422--430.
\end{thebibliography}
\end{document}